\documentclass[a4paper,12pt]{article}
\usepackage[nottoc]{tocbibind}
\usepackage{tikz}
\usetikzlibrary{arrows,topaths}
\usetikzlibrary{shapes}
\usetikzlibrary{cd}
\usetikzlibrary{decorations.pathmorphing}

\usepackage{graphicx}
\usepackage{verbatim}
\usepackage{blindtext}
\usepackage{mathtools}
\usepackage[mathscr]{euscript}
\usepackage{amssymb}
\usepackage{amsmath}
\usepackage{amsthm}
\usepackage{textcase}
\usepackage{color}
\usepackage[MeX]{polski} 
\usepackage[english]{babel}
\usepackage{lmodern}
\usepackage[utf8]{inputenc}
\usepackage{mathtools}
\usepackage{mathrsfs}
\usepackage{stackrel}

\usepackage{hyperref}
\hypersetup
{
	colorlinks=true,
	citecolor=black,
	filecolor=black,
	linkcolor=black,
	urlcolor=black,
}
\newcommand\N{\mathbb{N}}
\newcommand\K{\mathfrak{K}}
\renewcommand\L{\mathfrak{L}}
\newcommand\w{\omega}
\renewcommand\O{\mathcal{O}}
\newcommand\J{\mathcal{J}}
\newcommand\dol{\downarrow\!\!}

\DeclareMathOperator{\dom}{dom}
\DeclareMathOperator{\Dom}{Dom}
\DeclareMathOperator{\cod}{cod}

\DeclareMathOperator{\rng}{rng}

\DeclareMathOperator{\id}{id}
\newtheorem{thm}{Theorem}[section]
\newtheorem{cor}[thm]{Corrolary}
\newtheorem{lem}[thm]{Lemma}

\newtheorem{prop}[thm]{Proposition}
\theoremstyle{definition}
\newtheorem{Def}[thm]{Definition}
\newtheorem{rem}[thm]{Remark}
\newtheorem{exa}[thm]{Example}
\newtheorem{obs}[thm]{Observation}
\DeclarePairedDelimiter\abs{\lvert}{\rvert}
\DeclarePairedDelimiter\norm{\lVert}{\rVert}
\makeatletter
\let\oldabs\abs
\def\abs{\@ifstar{\oldabs}{\oldabs*}}
\let\oldnorm\norm
\def\norm{\@ifstar{\oldnorm}{\oldnorm*}}
\makeatother
\newcommand\Fraisse{Fra\"\i ss\'e}

\author{Szymon Głąb, Michał Pawlikowski}
\title{An inverse Fraïssé limit for finite posets and duality
for posets and lattices}
\date{ }

\providecommand{\keywords}[1]
{
  \small	
  \textit{Key words and phrases.} #1
}

\providecommand{\subjclass}[1]
{
  \small	
  2020 \textit{Mathematics Subject Classification. } #1
}

\begin{document}
\maketitle

\begin{abstract}
We consider a category of all finite partial orderings with quotient maps as arrows and construct a \text{\Fraisse} sequence in this category. Then we use commonly known relations between partial orders and lattices to construct a sequence of lattices associated with it. Each of these two sequences has a limit object -- an inverse limit, which is an object of our interest as well.

In the first chapter there are some preliminaries considering partial orders, lattices, topology, inverse limits, category theory and \text{\Fraisse} theory, which are used later.
In the second chapter there are our results considering a \text{\Fraisse} sequence in category of finite posets with quotient maps and properties of inverse limit of this sequence.
In the third chapter we investigate connections between posets and order ideals corresponding to them, getting an inductive sequence made of these ideals; then we study properties of the inverse limit of this sequence. 
\end{abstract}



\tableofcontents

\section{Introduction}
\text{\Fraisse} theory has its beginning in the Georg Cantor's work, where Cantor has shown that any countable, dense linear ordering without endpoints is isomorphic to $(\mathbb{Q},\leq)$. In the 1950s Roland \text{\Fraisse} expanded on the idea of Cantor and formulated it in the model theory framework \cite{fraisse}, which marked the beginning of a theory named after him. Some examples of \text{\Fraisse} limits are well-known objects, such as
$(\mathbb{Q},\leq)$ -- a limit of all linear orderings; the random graph -- a limit of all finite graphs; and some are completely new, like for example a limit of all finite partially ordered sets, called the random poset, as well as many more countable, ultrahomogeneous structures. In late 1970s James Schmerl in \cite{schmerl} characterized all possible countable ultrahomogeneous posets. Up to isomorphism, there are only countably many ultrahomogeneous countable posets, each of one of four types. Since ultrahomogeneous countable structures are \text{\Fraisse} limits of a finite structures family of the same type (this family is called age of the structure), Schmerl characterized all possible \text{\Fraisse} limits for finite posets.   

Due to Birkhoff Theorem (Theorem \ref{Birkhoff's_Theorem}) there is duality between finite posets and finite distributive lattices. Therefore one could expect the similar characterization to that of Schmerl for lattices. However Abogatma and Truss in \cite{AT} proved that there are continuum many pairwise non-isomorphic ultrahomogeneous lattices.  On the other hand if one demand families of finite lattices to be varieties, then by result \cite{DJ} of Day and Je\v{z}ek there are only three possible non-isomorphic \text{\Fraisse} limits.  

In XXI century new variants of \text{\Fraisse} theory were introduced by Sławomir Solecki, Trevor Irwin \cite{IS} and Wiesław Kubiś \cite{kubis2014,kubis}, and then further developed by other mathematicians. Limit objects in these theories can be uncountable and have found remarkable applications in topology and functional analysis.

In this article we work in generalised \text{\Fraisse} theory formulated by Wiesław Kubiś \cite{kubis2014} which uses category-theoretic approach. In this approach classical \text{\Fraisse} limit for posets is a direct limit of finite posets with embeddings. Here we are studying finite posets with \textbf{quotients}. Then the limit object is an uncountable poset with a natural topological structure. It turns out that this is the Cantor space with the extremely simple order structure, namely there are two kind of points: isolated or comparable with only one other point. The motivation for our investigation comes from \cite{GGK}, where the authors studied inverse \text{\Fraisse} limit for graphs.  

As we have mentioned finite lattices are dual to finite posets. Therefore we are also interested in finding dual limit to that of posets.  

The paper is organized as follows. In Section \ref{Section_Preliminaries} we present basic definitions and facts on posets, lattices and inverse limits, and gives an introduction to the categorical \text{\Fraisse} theory. In Section \ref{Section_Category_of_finite_posets} we build the model of the inverse \text{\Fraisse} limit for finite posets $\mathbb{P}$, and its properties. In Section \ref{Section_The_inverse_system_of_order_ideals} using duality between posets and lattices, having $\langle\{P_n\}_{n \in \N},\{p^n_k\}_{k<n}\rangle$  an inductive sequence consisting of finite posets with quotient maps and its limit $P=\varprojlim \langle\{P_n\}_{n \in \N},\{p^n_k\}_{k<n}\rangle$, we investigate an inverse limit $\mathbb{O}(P)$ of order ideals $\O(P_n)$. We show that $\mathbb{O}(P)$ is a $\sigma$-complete atomic lattice, and in Theorem \ref{InducedQuotientMapping} we show the connection between $\mathbb{P}$ and $\mathbb{O}(\mathbb{P})$. In Section \ref{Section_last} we define category of finite lattice in which $\mathbb{O}(\mathbb{P})$ is an inverse  \text{\Fraisse} limit, and present a different way of looking on order ideals $\O(P_n)$ and induced quotients $\hat{p}^n_k$.   

Definitions and facts concerning posets and lattices are taken from Davey and Priestley monograph \cite{priestley}.  Awodey's book \cite{awodey} and Kubiś's notes \cite{kubis} were used to gather the information on category theory and \text{\Fraisse} theory. Some topological facts are taken from Srivastava's monograph \cite{srivastava}.

\section{Preliminaries}\label{Section_Preliminaries}

We will use $\mathbb{N}$ to denote the set of natural numbers starting from $1$ and we will use $\w$ to denote the set of natural numbers starting with $0$.
\subsection{Definitions and facts concerning posets}
In this section we present terminology and notation that we are using in the sequel. Most of them come from \cite{priestley}. 

Let $A$ be a set and $\leq$ be a binary relation on $A$ that is reflexive, transitive and antisymmetric. We say that structure $(A,\leq)$ is a \emph{poset}.
For $x, y \in A$ if neither $x\leq y$ nor $y\leq x$, then we say that $x$ and $y$ are incomparable, in symbols $x\parallel y$. 
We say that a poset $(A,\leq)$ is a \emph{chain}, if for all $x,y \in A$, $x\leq y$ or $y \leq x$. If $A=\{1,\dots,n\}$ is considered with the natural order, then we denote such chain by $\mathbf{n}$.
We say that a poset $(A,\leq)$ is an \emph{antichain}, if for all $x,y \in A$, $x\leq y \Leftrightarrow x=y$. If $A=\{1,\dots,n\}$ is ordered by equality, then we denote such antichain by $\overline{\mathbf{n}}$.

Let $(A;\leq_{A}),(B;\leq_{B})$ be posets and $\phi$ a mapping $\phi:A \rightarrow B$. If for all $x,y \in A$ such that $x \leq_{A} y$, $\phi(x) \leq_{B} \phi(y)$ then we call $\phi$ an \emph{order homomorphism} (we will often abbreviate it to \emph{homomorphism}). If $\phi(x) \leq_B \phi(y) \iff x \leq_A y$ for all $x,y \in A$, then we say that $\phi$ is an \emph{order-embedding}. Furthermore, if it is onto, we say that it is an \emph{order-isomorphism}. We say that a homomorphism $\phi$ is a \emph{quotient map} if it is onto and for every $p,r \in B$ such that $p\leq_{B}r$ there are $x,y \in A$ such that $\phi(x)=p,\phi(y)=r$ and $x \leq_{A} y$.

Let $(A;\leq_{A}),(B;\leq_{B})$ be disjoint posets. We define the \emph{linear sum} $A \oplus B$ by taking the following order relation on $A \cup B: x \leq y$ if and only if one of the following conditions is met
\begin{enumerate}
\item $x,y \in A$ and $x\leq_{A}y$
\item $x,y \in B$ and $x\leq_{B}y$
\item $x \in A$ and $y \in B$
\end{enumerate}
The \emph{disjoint union} $A \dot \cup B$ is the ordered set $A \cup B$ formed by defining $x \leq y$ if and only if one of the following conditions is met
\begin{enumerate}
\item $x,y \in A$ and $x \leq_{A} y$
\item $x,y \in B$ and $x \leq_{B} y$
\end{enumerate}
If $A$ and $B$ are not disjoint, we can define their disjoint copies, for example $A\times\{0\}$ and $B\times\{1\}$, endow them with order structures isomorphic to the original ones, and then take their linear sum and disjoint union. Having this in mind we will write, for example, $\mathbf{2}\oplus\mathbf{2}=\mathbf{4}$ and $\overline{\mathbf{2}}\dot\cup\overline{\mathbf{2}}=\overline{\mathbf{4}}$  which is true up to isomoprhism.

Let $P$ be a poset and $Q \subset P$. We say that $Q$ is a \emph{down-set} if for every $q \in Q$ and for every $p \in P$ such that $p \leq q$ we have that $p \in Q$. By $\dol p$ we shall denote the following set $\lbrace m \in P: m \leq p \rbrace$. We say that it is a \emph{principal} down-set.
It is easy to see that $\dol p$ is a down-set for every $p \in P$.
By $\O(P)$ we denote a partial ordering of all down-sets of $P$ with $\subset$ as an order relation. Set $\O(P)$ is called an \emph{order ideal}.
Let $S \subset P$. An element $x \in P$ is called an \emph{upper bound} of $S$ if for every $y \in S$ we have $y \leq x$. Similarly $x$ is the \emph{lower bound} of $S$ if $y \geq x$ for every $y \in S$. The \emph{supremum} of $S$ is the least upper bound of $S$ and we denote it by $\sup(S)$. Similarly the \emph{infimum} of $S$ is the greatest lower bound of $S$ and we denote it by $\inf(S)$.
Supremum and infimum of $S$ may not exist even for finite sets. However, if they exist, they must be unique.

Let $P$ be a poset. We say that $P$ is a \emph{lattice} if for every $x, y \in P$ $\sup (\lbrace x, y \rbrace) \text{ and } \inf (\lbrace x, y \rbrace)$ exist. Usually we write $\sup( \lbrace x, y \rbrace)$, $\inf (\lbrace x, y \rbrace)$ as $x \vee y$, $x \wedge y$ respectively.

Supremum and infimum are associative, i.e. $x\lor(y\lor z)=(x\lor y)\lor z$ and $x\land(y\land z)=(x\land y)\land z$, and commutative. If a lattice has the minimal element, then we call it as zero element and denote it by 0. Similarly the maximal element is denoted by 1. Note that every finite lattice has both minimal and maximal elements. 

Let $L$ be a lattice. We say that $L$ is \emph{distributive} if for all $x,y,z \in L$ the following is true:
\begin{equation}
x \wedge (y \vee z)=(x \wedge y) \vee (x \wedge z)
\end{equation}

Let $P$ be a poset. It is easy to see that $\O(P)$ is a lattice with $\sup\lbrace x, y \rbrace= x \cup y$ and $\inf \lbrace x, y \rbrace= x \cap y$.
\begin{lem} \cite[p. 46]{priestley}\label{LemmaOnSupremumOfUnions}
	Let $\{A_i\}_{i \in I}$ be a family of sets contained in a lattice $L$, such that for every $i \in I$, $\sup(A_i)$ exists and $\sup(\bigcup_{i \in I} A_i)$ exists. Then $\bigvee_{i \in I} \sup(A_i):=\sup\{\sup(A_i):i\in I\}$ exists and $\sup(\bigcup_{i \in I} A_i)=\bigvee_{i \in I} \sup(A_i)$.
\end{lem}

Let $L$ be a lattice. An element $x \in L$ is \emph{join-irreducible} if the following conditions are met
\begin{enumerate}
\item[(i)] $x \neq 0$ (in case $L$ has a zero)
\item[(ii)]$x=a \vee b$ implies $x=a$ or $x=b$ for all $a, b \in L$.
\end{enumerate}
We denote the set of all join-irreducible elements of $L$ by $\J(L)$. This set inherits $L$'s order relation.

The simple example of elements which are join-irreducible and those which are not is given after the Definition \ref{definitionAtom}. Now we present a relation between $\leq$ order on $P$ and $\subset$ on $\O(P)$.
\begin{lem}\cite[p. 21]{priestley}\label{DownsetLemma1}
Let $P$ be a poset and $x, y \in P$. The following statements are equivalent:
\begin{enumerate}
\item[\emph{(i)}] $x \leq y$
\item[\emph{(ii)}] $\dol x \subset \dol y$;
\item[\emph{(iii)}] For every $Q \in \O(P)$ if $y \in Q$ then $x \in Q$.
\end{enumerate}
\end{lem}
Let $L$ and $K$ be lattices. A map $f:L \to K$ is said to be a \emph{lattice homomorphism} if $f$ is join-preserving and meet-preserving, that is, for all $a,b \in L$,
$$ f(a \vee b)=f(a) \vee f(b) \text{ and } f(a \wedge b)=f(a) \wedge f(b).$$
A bijective homomorphism is a lattice isomorphism.
\begin{lem}\cite[p. 44]{priestley}\label{Lattice_And_Order_Isomorphisms}
Let $L$ an $K$ be lattices and $f:L \to K$ a map. Then $f$ is a lattice isomorphism if and only if it is an order-isomorphism.
\end{lem}

\begin{lem}\label{DownsetLemma2}
Let $P$ be a finite poset and $Q \subset P$ be a down-set. Then $Q= \bigcup_{k=1} ^n  \dol x_k$ where $x_1, \ldots, x_n$ are pairwise distinct maximal elements of $Q$ (note that $Q$ has maximal elements, as it is a finite poset). Furthermore it is the only representation of $Q$ as a union of pairwise incomparable principal down-sets. We shall call the set $\{\dol x_k:k \leq n\}$ a canonical decomposition of $Q$.
\end{lem}
\begin{proof}
Let $Q$ be a down-set and $x_1, \ldots, x_n$ be a list of all maximal elements of $Q$.
It is easy to see that $\bigcup_{k=1} ^n \dol x_k \subset Q$ because $x_1, \ldots, x_n \in Q$ and $Q$ is a down-set. Now if we take $q \in Q$ such that $q$ is not maximal then, as $Q$ is finite, we will find maximal element $x_i$ such that $q \leq x_i$, so $q \in \bigcup_{k=1} ^n \dol x_k$. So $Q= \bigcup_{k=1} ^n \dol x_k$. What follows from Lemma \ref{DownsetLemma1} and maximality of $x_l$'s is that $\dol x_i \parallel \dol x_j$ for $i \neq j$.
Now we prove that it is the only pairwise incomparable representation of $Q$.

Assume that $Q= \bigcup_{k=1} ^n  \dol x_k = \bigcup_{k=1} ^m  \dol y_k$, with $ \dol y_i \parallel \dol y_j \text{ for } i \neq j$. Clearly $n \leq m$ because otherwise we have that for some $\dol x_i,\dol x_j$ we find $\dol y_p$ such that $\dol x_i \subset \dol y_p,\dol x_j \subset \dol y_p$ so $x_i, x_j$ would not be distinct. The fact that we will find such $\dol y_p$ is because, if $\dol x_k \subset \bigcup_{k=1} ^m \dol y_k$, then $x_k \in \bigcup_{k=1} ^m \dol y_k$, thus $x_k \in \dol y_p$ for some $p \in \{1, \dots, m\}$, so $x_k \leq y_p$. It follows that $\dol x_k \subset \dol y_p$.

Now it is easy to see that for every $\dol x_k$ we find $\dol y_{\ell}$ such that $\dol x_k \subset \dol y_{\ell}$. On the other hand for every $\dol y_{\ell}$ we find $\dol x_p$ such that $\dol y_{\ell} \subset\dol x_p$. So $\dol x_k \subset \dol y_{\ell} \subset \dol x_p$. Clearly $k=p$, since $x_k$ and $x_p$ are comparable. Thus $y_{\ell}=x_k$, and consequently $\{x_1,\dots,x_n\}\subset\{y_1,\dots,y_m\}$. Suppose that $m>n$. Then there is $\ell$ with $y_\ell\notin\{x_1,\dots,x_n\}$. But then by maximality of $x_i$'s, $y_\ell\leq x_j$ for some $j$. This contradicts the assumption that $y_k$'s are pairwise incomparable.  
So $m=n$ and $\lbrace y_1, \ldots, y_n \rbrace = \lbrace x_1, \ldots, x_n \rbrace$.
\end{proof}
From the above proof we extract the following fact.
\begin{lem}\label{CanonicalDecompositionLemma}
Let $P$ be a finite poset, $Q \subset P$ be a down-set. Let $\{\dol x_k: k \leq n \}$ be the canonical decomposition of $Q$ and let $\{\dol y_k: k \leq m \}$ be such that $Q= \bigcup_{k=1} ^m \dol y_k$. Then $\{\dol x_k: k \leq n \} \subset \{\dol y_k: k \leq m \}$.
\end{lem}

\begin{Def}\cite[p. 113]{priestley}\label{definitionAtom}
	Let $L$ be a lattice with the least element $0$ and $a \in L$. We say that $a$ is an \emph{atom} in $L$, if $0 < a$ and if $x < a$, then $x=0$.
	A lattice is called \emph{atomic}, if for every $x \in L, x \neq 0$ we can find an atom $a$ such that $a \leq x$.
\end{Def}
To illustrate the notion of atom, let us consider the power set $\mathcal{P}(\{0,1\})$ with the inclusion as a an order relation. This is a lattice with the least element $\emptyset$. Then $\{0\}$ and $\{1\}$ are atoms while $\{0,1\}$ is not. Note also that $\{0\}$ and $\{1\}$ are join-irreducible while $\{0,1\}$ is not. 
\begin{lem}\cite[p. 117]{priestley}\label{Lemma_On_Irreducible_Elements}
Let $L$ be a distributive lattice and let $x \in L$, with $x \neq 0$ in case $L$ has a zero. Then the following are equivalent:
\begin{enumerate}
	\item[(i)] $x$ is join-irreducible
	\item[(ii)] if $a,b \in L$ and $x \leq a \vee b$, then $x \leq a$ or $x \leq b$
	\item[(iii)] for any $k \in \N$, if $a_1,\dots, a_k \in L$ and $x \leq a_1 \vee a_2 \vee \dots \vee a_k$, then $x \leq a_i$ for some $i \in \{1, \dots, k\}$.
\end{enumerate}
\end{lem}
\begin{lem}\cite[p. 116]{priestley}\label{Join_Irreducible_Elements}
Let $P$ be a poset. Then $\J(\O(P))=\{\dol x: x \in P\}$
\end{lem}

\begin{lem}\cite[p. 116]{priestley}\label{Isomorphism_Of_P_And_PD}
Let $P$ be a poset and $\phi:P \to \J(\O(P))$ be defined as $\phi(x)=\dol x$. Then $\phi$ is an order-isomorphism. In other words $\phi$ is an order-embedding of $P$ into $\O(P)$.
\end{lem}

\begin{thm}\cite[p. 118]{priestley}[Birkhoff's Theorem] \label{Birkhoff's_Theorem}
Let $L$ be a finite, distributive lattice. Then the map $\eta: L \to \O(\J(L))$ defined by
$$ \eta(a):= \{x \in \J(L): x \leq a\}$$
is an isomorphism of $L$ onto $\O(L)$.
\end{thm}

\subsection{Topological definitions and facts concerning the inverse limits}
In this section we present the definition of the inverse limit and the notion of profinite poset.
Let $X$ be a topological space. We say that $X$ is \emph{$0$-dimensional}, if it has a base of clopens.

Let $P,Q$ be posets with topology $\tau, \sigma$ respectively. We say that $f: P \rightarrow Q$ is a \emph{topological isomorphism} if it is a homeomorphism and an order-isomorphism.

Let $P_n$ be non-empty, finite sets with the discrete topology. Let $E_n$ be relations on $P_n$. We equip $P_1 \times P_2 \times \dots$ with the product topology.
Sets of the form $$\{x_1\} \times \{x_2\} \times \dots \times \{x_n\} \times P_{n+1} \times P_{n+2} \times \dots$$ form a base for topology on $P_1 \times P_2 \times \dots$. Moreover, $P_1 \times P_2 \times \dots$ is $0$-dimensional, compact and metrizable.

Now assume we have a family of onto homomorphisms $\{p_k ^m:m, k \in \N, k<m\}$ where $p_k ^m:P_m \rightarrow P_k, m \in \N, k <m$, such that $ p_\ell ^{k} \circ p_k ^{n}=p_\ell ^{n}$ for $\ell < k < n$. An inverse limit of $P_n$ with $p_k ^m$ is a set $$\varprojlim \left< \{P_n\}_{n \in \N},\{p^m_k \}_{k < m}\right>:=\left\{(x_1, x_2, \dots) \in \prod_{i \in \N} P_i:p_n ^{n+1}(x_{n+1})=x_n \text{ for each } n \in \N \right\}$$

We define relation $E$ on inverse limit as follows:
$$(x_n)E(y_n) \iff x_n E_n y_n \text{ for each } n \in \N $$
If relations $E_n$ are partial orders, then $E$ is also a partial ordering. Then the inverse limit with this ordering (or structure topologically isomorphic to it) is called a \emph{profinite poset}.
\begin{lem}\label{Inverse_Limit_Closed}
$P:=\varprojlim \left< \{P_n\}_{n \in \N},\{p^m_k \}_{k < m}\right>$ is closed.
\end{lem}
\begin{proof}
We will show that $(\prod_{i \in \N} P_i) \setminus P$ is open. Let $(x_n) \in (\prod_{i \in \N} P_i) \setminus P$ and $n$ be the smallest index such that $p_n ^{n+1}(x_{n+1}) \neq x_n$. Then see that $(x_n) \in \{x_1\} \times \{x_2\} \times \dots \times \{x_{n+1}\} \times P_{n+2} \times \dots \subset (\prod_{i \in \N} P_i) \setminus P$, so $(\prod_{i \in \N} P_i) \setminus P$ is open, thus $P$ is closed.
\end{proof}

\begin{lem}\label{RelationOnInverseLimitIsClosed}
    Relation $E$ is a closed set in $P \times P$.
\end{lem}
\begin{proof}
We will show that $P^2 \setminus E$ is an open set. Take $((x_n), (y_n)) \in P^2 \setminus E$.  Let $n$ be the smallest index such that $ \neg (x_n E_n y_n$). 

Then $((x_n), (y_n)) \in (P \cap \{x_1\} \times \{x_2\} \times \dots \times \{x_n\} \times P_{n+1} \times \dots) \times (P \cap \{y_1\} \times \{y_2\} \times \dots \times \{y_n\} \times P_{n+1} \times \dots) \subset P^2 \setminus E$ and $(P \cap \{x_1\} \times \{x_2\} \times \dots \times \{x_n\} \times P_{n+1} \times \dots) \times (P \cap \{y_1\} \times \{y_2\} \times \dots \times \{y_n\} \times P_{n+1} \times \dots)$ is open in $P^2$, thus $P^2 \setminus E$ is open, which makes $E$ closed.
\end{proof}
Lemma \ref{Inverse_Limit_Closed} proves that our inverse limit is closed and it is a sub-space of compact space, thus it is compact. See that sets  $P \cap (\{x_1\} \times \dots \times \{x_n\} \times P_{n+1} \times \dots)$ form a base for the inverse limit, thus it is $0$-dimensional and it is also metrizable.
\subsection{Definitions and facts from the category theory}
In this section we present definitions and facts concerning category theory and $\text{\Fraisse}$ sequences. They come from \cite{awodey}, \cite{kubis} and \cite{kubis2014}.

A \emph{category} is a structure of the form $\K =\langle V,A,\dom,\cod,\circ \rangle$ where $V$ is a class of objects, $A$ is a class of arrows, $\dom:A \rightarrow V$ and $\cod:A \rightarrow V$ are the domain and codomain functions, and $\circ$ is a partial binary operation on arrows, such that the following conditions are satisfied:
\begin{enumerate}
\item $f \circ g$ is defined whenever $\cod(g)=\dom(f)$ and $\circ$ is associative, that is, $(f \circ g) \circ h=f \circ (g \circ h)$ whenever $\cod(h)=\dom(g)$ and $\cod(g)=\dom(f)$
\item For each object $a \in V$ there is an arrow $\id_a$ with $\dom(\id_a)=\cod(\id_a)=a$ and such that $\id_a \circ g=g,f \circ \id_a=f$ holds for every $f,g \in A$ with $\cod(g)=\dom(f)=a$.
\item For every objects $a,b \in V$ the class \[ \K(a,b):=\{f \in A: \dom(f)=a \wedge \cod(f)=b \} \] is a set.
\end{enumerate}

Let $\K$ be a category. We say that $a$ is an \emph{initial object}, if for every object $b$ in $\K$ there is a unique arrow $f:a \rightarrow b$. $a$ is a \emph{terminal object} if for every object $b$ in $\K$ there is a unique arrow $f:b \rightarrow a$.

Let $\K$ be a category. We say that $\K$ is \emph{directed} if for every $a,b \in \K$ there exists $c \in \K$ such that both sets $\K(a,c),\K(b,c)$ are not empty.

Let $\K$ be a category. We say that $\K$ has the \emph{amalgamation property} if for every $a,b,c \in \K$ and for every morphisms $f \in \K(a,b),g \in \K(a,c)$ there exists $d \in \K$ and morphisms $f^{'} \in \K(b,d),g^{'} \in \K(c,d)$ such that $f^{'} \circ f=g^{'} \circ g$. In other words the diagram 
\begin{equation}
\begin{tikzcd}
									& b \arrow[rd, "f^{'}"]&		\\							
a \arrow[ru, "f"] \arrow[rd, "g"'] &					 		    &d \\
									& c \arrow[ru, "g^{'}"'] &		\\					
\end{tikzcd}
\end{equation}
commutes.

Let $\K$ be a category. A pair of the form $\langle \lbrace u_n \rbrace,\lbrace u_k ^{n} \rbrace_{k<n} \rangle$ such that $\lbrace u_n :n\in\N\rbrace \subset \K$, $u_k ^{n} \in \K(u_k, u_n)$ for every $n, k \in \N$, where $k<n$, and $ u_k ^{n} \circ u_\ell ^{k}=u_\ell ^{n}$ is called an \emph{inductive sequence} in $\K$ and we denote it by $\overrightarrow{u}$.
\footnote{An inductive sequence can be enumerated by any ordinal. The general definition sounds as follows.
Let $\K$ be a category and $\delta >0$ an ordinal. A pair of the form $\langle \lbrace a_\xi \rbrace_{\xi<\delta},\lbrace a_\xi ^{\eta} \rbrace_{\xi<\eta<\delta} \rangle$ such that $\lbrace a_{\xi} :\xi<\delta \rbrace \subset \K$ and $ a_\eta ^{\varrho} \circ a_\xi ^{\eta}=a_\xi ^{\varrho}$ is called an \emph{inductive sequence} in $\K$ and we denote it by $\overrightarrow{a}$. The ordinal $\delta$ is the \emph{length} of $\overrightarrow{a}$.}

\hypertarget{ExtensionProperty}{}

Let $\overrightarrow{u}$ be an inductive sequence in category $\K$.
We say that it has the \emph{extension property} (E) if for every arrows $f:a \to u_n, g:a \to b$ where $n \in \N$, there exists $m > n$ and $h \in \K(b, u_m)$ such that the diagram 
\begin{equation*}
    \begin{tikzcd}
    u_n \arrow[r, "u_n ^m"] & u_m \\
    a \arrow[u, "f"] \arrow[r, "g"'] & b \arrow[u, "h"'] \\
    \end{tikzcd}
\end{equation*}
commutes.

Let $\K$ be a fixed category and let $\overrightarrow{u}$ be an inductive sequence in that category. We say that $\overrightarrow{u}$ is a \emph{\text{\Fraisse} sequence} if the following conditions are met
\begin{enumerate}
\item[(U)] \hypertarget{UProperty}{} for every $x \in \K$, there exists $n \in \N$ such that $\K(x,u_n) \neq \emptyset$ ;
\item[(A)] \hypertarget{AProperty}{} for every $k \in \N$ and for every arrow $f \in \K(u_k,y)$, where $y \in \K$, there exist $\ell>k$ and $g \in \K(y,u_\ell)$ such that $u^{\ell}_{k} = g \circ f$. 
\end{enumerate}

Let $\K$ be a fixed category and $\mathcal{F}$ a family of arrows. By $\Dom(\mathcal{F}) $ we denote the set $\{\dom(f) \in \K:f \in \mathcal{F}\}$. We say that $\mathcal{F}$ is \emph{dominating} in $\K$ (or that $\K$ is \emph{dominated} by $\mathcal{F}$) if the following conditions are met
\begin{enumerate}
\item[(D1)] \hypertarget{D1Property}{} for every $x \in \K$ there is $a \in \Dom(\mathcal{F})$ such that $\K(x,a) \neq \emptyset$ ;
\item[(D2)] \hypertarget{D2Property}{} given $a \in \Dom(\mathcal{F})$ and $f\in\K(a,y)$, there exist $g\in\K(y,b)$ such that $g \circ f \in \mathcal{F}$.
\end{enumerate}

\begin{lem}\label{InitialLemma}
Let $\K$ be a category with an initial object $a$. If $\K$ has the amalgamation property, then $\K$ is directed.
\end{lem}
\begin{lem}\label{APropertyImpliesUProperty}
Let $\overrightarrow{u}$ be an inductive sequence in directed category $\K$. If $\overrightarrow{u}$ has \hyperlink{AProperty}{\emph{(A)}} property, then it has \hyperlink{UProperty}{\emph{(U)}} property.
\end{lem}
\begin{lem}\label{FraisseSequenceHasEProperty}
Let $\overrightarrow{u}$ be a \text{\Fraisse} sequence in category $\K$ which has the amalgamation property. Then $\overrightarrow{u}$ has \hyperlink{ExtensionProperty}{\emph{(E)}} property.
\end{lem}
\begin{thm}[Existence of Fra\"\i ss\'e sequence]\label{FraisseExistence}
Let $\K$ be a directed category with amalgamation property. Assume that $\K$ is dominated by a countable family of arrows $\mathcal{F}$. Then $\K$ has a Fra\"\i ss\'e sequence.
\end{thm}

\section{Category of finite posets}\label{Section_Category_of_finite_posets}

Let $V$ be a class of all finite posets (we identify all isomorphic posets with each other). Then $V$ is a countable set. Formally, one can define $V$ as a set of all finite posets on $\N$. 
Now let $A$ be the class of all quotient mappings between finite posets; we say that $f$ is an arrow from $a$ to $b$, provided $f$ is quotient mapping from $b$ onto $a$. Then $\dom f$ is the range of $f$ and $\cod f$ is the domain of $f$.   
The composition $g\bullet f$ of two arrows $f$ from $a$ to $b$ and  $g$ from $b$ to $c$ is $f \circ g:c\to a$, where $\circ$ is the usual composition for mappings. 
The composition of two quotient mapping is a quotient mapping, identity is a quotient mapping and class of arrows from $a$ to $b$ is a subset of $a^b$. Therefore $\K:=\langle V,A,\dom,\cod,\bullet \rangle$ is a category.

In the sequel we will use $\leq$ for poset relation if it is clear what is its domain. We will use subscripts, say $\leq_A$, $\leq_B$, to indicate relations domains if more than one poset is currently considered.

Let $(A, \leq)$ be a poset and let $x,y \in A$. If there is no $z$ less than $x$, no $z$ greater than $y$ and $x<y$, then we say pair $(x,y)$ is a \emph{$\mathbf{2}$-component} in $A$.

As $\K$ is a countable category, if we prove that it it directed and has the amalgamation property, then from Theorem \ref{FraisseExistence} we get, that there is a \text{\Fraisse} sequence in that category. In the next section we are going to prove that our category indeed has these two properties.

\subsection{Directedness and the amalgamation property}
Let $E$ be a relation on $X$. We say that $x$ is $E$-isolated, if for every $y \neq x$ neither $xEy$ nor $yEx$.

\begin{lem}
Let $(A,\leq_A)$ be a finite poset without $\leq_A$-isolated points. There is a quotient map $\varphi:\dot\bigcup_{i < n} \mathbf{2} \rightarrow A$ for some $n \in \mathbb{N}$.
\end{lem}
\begin{proof}
Let $(x_0,y_0), \ldots, (x_{m-1},y_{m-1})$ be an enumeration of all pairs $(a,b) \in A \times A$ with $a <_{A} b$. By $(B,\leq_B)$ we denote the set $\lbrace 0, 1, \ldots, 2m-1\rbrace$ with a partial order given by $x \leq_B y \Leftrightarrow (x=y \vee \exists_{i<m} (x=2i \wedge y=2i +1))$. We define $f:B \rightarrow A$ as follows 
\begin{equation}
f(2i)=x_i,f(2i+1)=y_i \text{ for } i<m
\end{equation}
Clearly $f$ is a quotient map. Since $B$ and $\dot\bigcup_{i<m} \mathbf{2}$ are isomorphic, the desired mapping $\varphi$ is a composition of $f$ with the isomorphism. 
\end{proof}
\begin{lem}\label{sticks}
Let $(C,\leq_C)$ be a poset. There is a quotient map $\chi:\dot\bigcup_{i< p} \mathbf{2} \rightarrow C$ for some $p \in \mathbb{N}$.
\end{lem}
\begin{proof}
Let $I$ be a set of all isolated points in $C$, and let $A=C \setminus I$. Then there is a quotient map $\phi :\dot\bigcup_{j<n} \mathbf{2} \rightarrow A$ for some $n \in \mathbb{N}$. Let $i$ be the cardinality of $I$. We define $\psi: \dot\bigcup_{j<i} \mathbf{2} \rightarrow I$ so that we map $k$-th $\mathbf{2}$-component of $\dot\bigcup_{j<i} \mathbf{2}$ onto $k$-th point of $I$. Finally we define $\chi: \dot\bigcup_{j<i} \mathbf{2} \text{  }\dot\cup \text{  } \dot\bigcup_{j<n} \mathbf{2} \rightarrow C$ as $\chi := \phi \cup \psi$.
\end{proof}
\begin{thm}\label{DirectednessANDAmProp}
Category of finite posets is directed and has the amalgamation property.
\end{thm}
\begin{proof}
Firstly we show that this category has the amalgamation property.
Let $A, B$ and  $C$ be posets and let $f: B \to A$ and $g: C \to A$ be quotient maps. By Lemma \ref{sticks} there are $B^{'} \text{ and } C^{'}$, disjoint unions of $\mathbf{2}$'s, and $f^{'}:B^{'}  \to B \text{ and } g^{'}: C^{'}  \to C$ quotient mappings. If we find a poset $D$ and quotients $q: D  \to B^{'} \text{ and } p: D  \to C^{'}$ such that the diagram \begin{equation}
\begin{tikzcd}
    & B \arrow[ld, "f"']  && B^{'} \arrow[ll, "f^{'}"']    \\
A&&															&&D	\arrow[lu, "q"']	\arrow[ld, "p"] \\
	& C \arrow[lu, "g"]	&& C^{'} \arrow[ll, "g^{'}"] \\	  
\end{tikzcd}
\end{equation}
commutes, then we are done. 
Let $D:=B' \dot\cup C'$. We define $q: D \rightarrow B'$ as follows: if $z \in B'$ then $q(z)=z$. If $z \in C'$ then let $v \in C'$ be such that $z \leq_D v$ or $v \leq_D z$. We define $q$ at $z$ and $v$ at the same time, so we may assume that $z<_D v$.
\begin{enumerate}
\item[(i)] If $g\circ g'(z)=g\circ g'(v)$, then because $f$ is onto we find $x \in B'$ such that $f\circ f'(x)=g\circ g'(z)=g\circ g'(v)$. We put $q(z)=q(v)=x$.
\item[(ii)] If $g\circ g'(z) \neq g\circ g'(v)$, then $g\circ g'(z)<_A g\circ g'(v)$. Because $f$ is a quotient map we find $x,y \in B$ such that $x <_B y$ and $f\circ f'(x)=g\circ g'(z),f\circ f'(y)=g\circ g'(v)$. We put $q(z)=x,q(v)=y$.
\end{enumerate}
Defining $p$ in a similar fashion makes that $f \circ f'\circ q=g \circ g'\circ p$ and both $q,p$ are quotient maps.
So the functions $f'\circ q$ and $g'\circ p$ witness the amalgamation property in considered category. It is clear that $\mathbf{1}$, that is the one-element poset, is an initial object of our category. By Lemma \ref{InitialLemma} the category is directed.
\end{proof}

From Theorem \ref{DirectednessANDAmProp} and Theorem \ref{FraisseExistence} we get that there is a \text{\Fraisse} sequence in the category of finite posets. Now we are going to provide its direct construction.

\subsection{A universal projective model of profinite poset}\label{UniversalProjectivePoset}

The construction in this section is an adaptation of that presented in \cite{GGK} for graphs.  

Let $\mathbb{P}$ be a set $\lbrace 0,1,2,3 \rbrace ^{\w}$ with a partial order $x \leq y \Leftrightarrow x=y$ or one of the following conditions is met:
\begin{enumerate}
\item[(i)${}_\w$] $\exists_n \left[\left(\forall_{k<n} \text{ } x(k)=y(k)\right) \wedge \left(x(n)=2,y(n)=3\right) \wedge \left(\forall_{k>n} \text{ } x(k)=y(k) \in \lbrace 0,1 \rbrace \right)\right]$
\item[(ii)${}_\w$] $\left(x(0)=0,y(0)=1 \right) \wedge \left(\forall_{k \geq 1} \text{ } x(k)=y(k) \in \lbrace 0,1 \rbrace \right)$
\end{enumerate}
Note that this is a partial order. Let $P_n := \lbrace 0,1,2,3 \rbrace ^{n}$ for $n \in \N$. We define a partial ordering $x \leq_n y \Leftrightarrow x=y$ or one of the following conditions is met:
\begin{enumerate}
\item[(i)${}_n$] $\exists_{l<n}\left[\left( \forall_{k<l} \text{ } x(k)=y(k)\right) \wedge \left(x(l)=2,y(l)=3\right) \wedge \left(\forall_{k \in (l,n)} \text{ } x(k)=y(k) \in \lbrace 0,1 \rbrace \right)\right]$
\item[(ii)${}_n$] $\left(x(0)=0,y(0)=1 \right) \wedge \left(\forall_{k \in [1,n)} \text{ } x(k)=y(k) \in \lbrace 0,1 \rbrace \right)$
\end{enumerate}
The fact that two distinct $x,y\in P_n$ are in the order implies that they are equal at all but one coordinate. Moreover, there are no isolated points with respect to the order -- for any $x\in P_n$ we find distinct $y\in P_n$ comparable with $x$ by changing one appropriate coordinate of $x$. Therefore $P_n$ is isomorphic to $\dot\bigcup_{i < 2 \cdot 4^{n-1} } \mathbf{2}$. 
We define $p_n:P_{n+1} \to P_n$ such that $p_n(x):=x_{\vert_n}=(x(0),\dots,x(n-1))$ for $n \in \mathbb{N}$. Let us observe the following:
\begin{lem}
$p_n$ is a quotient map for all $n \in \N$.
\end{lem}
\begin{proof}
Firstly we show that it is a homomorphism. Let $x \leq_{n+1} y,x \neq y$. We consider three cases.\begin{enumerate}
\item[(i)] If $x_{\vert_n}=y_{\vert_n}$ then $p_n(x)=x_{\vert_n} \leq_n y_{\vert_n}=p_n(y)$
\item[(ii)] If $x(0)=0,y(0)=1$ and $\forall_{k \in [1,n +1)} \text{ } x(k)=y(k) \in \lbrace 0, 1 \rbrace$ then $p_n(x)=x_{\vert_n} \leq_n y_{\vert_n}=p_n(y)$
\item[(iii)] If $x(l)=2,y(l)=3 \text{ for some } l<n+1$ then $\forall_{k <l} \text{ } x(k)=y(k) \text{ and } \forall_{k \in (l,n+1)} \text{ } x(k)=y(k) \in \lbrace 0, 1 \rbrace$, so $p_n(x)=x_{\vert_n} \leq_n y_{\vert_n}=p_n(y)$. 
\end{enumerate}
Thus $p_n$ is a homomorphism. Since $p_n$ is onto, we only need to show that it is a quotient map. Let $a \leq_n b$ and $a,b \in P_n$. We define $a':=(a(0), \ldots, a(n-1), 0), b':=(b(0), \ldots, b(n-1), 0)$. Then $a' \leq_{n+1} b'$ and $p_n(a')=a,p_n(b')=b$.
\end{proof}

If $n>k$ then $p_k ^{n}=p_k \circ p_{k+1} \circ \ldots \circ p_{n-1}$ is the restriction to the first $k$ coordinates. Clearly $p_k ^{n+1}=p_k ^{n} \circ p_n$. We consider the sequence of posets 

\begin{equation}\label{SequanceOfUniwersalProjectivePosets}
\begin{tikzcd}
P_1 && \arrow[ll, "p_1"'] P_2 && \arrow[ll, "p_2"'] P_3 && \arrow[ll, "p_3"'] \ldots \\
\end{tikzcd}
\end{equation}

Note that, if we show that $\langle \{P_n\}_{n \in \N}, \{p_k ^n\}_{k<n} \rangle$ is a \text{\Fraisse} sequence in category of finite posets, then the inverse limit  $ \varprojlim \langle \{P_n\}_{n \in \N}, \{p_k ^n\}_{k<n} \rangle$ is the inverse $\text{\Fraisse}$ limit in this category. 
To show that $\mathbb{P}$ is topologically isomorphic to the inverse limit of $P_n$'s we must show the following.

\begin{lem}\label{OrderingsAreTheSameOnPandOnInverseLimit}
$x \leq y$ if and only if $p_n(x) \leq_n p_n(y) \text{ for all } n \in \mathbb{N}$.
\end{lem}
\begin{proof}
The \textit{only if} part of equivalence is immediate. 
So let us assume that $p_n(x) \leq_n p_n(y) \text{ for all } n \in \mathbb{N}$. If $p_n(x) = p_n(y) \text{ for all } n \in \mathbb{N}$ then $x=y$. Otherwise we choose $m:=\min \lbrace k \in \mathbb{N}: p_k(x) \neq p_k(y) \rbrace$ and consider the following cases:
\begin{enumerate}
\item[(i)]  If $m>1$ then for all $j<m-1$ we have $x(j)=y(j)$, $x(m-1)=2,y(m-1)=3$ and for all $j>m-1$ we have $x(j)=y(j) \in \lbrace 0, 1 \rbrace$.
\item[(ii)] If $m=1$ then $x(0)=0,y(0)=1$ and for all $j \geq 1$ we have $x(j)=y(j) \in \lbrace 0, 1 \rbrace$ or $x(0)=2,y(0)=3$ and for all $j \geq 1$ we have $x(j)=y(j) \in \lbrace 0, 1 \rbrace$.
\end{enumerate}
In both cases $x \leq y$.
\end{proof}
\begin{thm}\label{PIsTopIsoToInvFrLim}
Let $\phi:\mathbb{P} \to \varprojlim \langle  \{P_n\}_{n \in \N},\{p^n_k \}_{k < n} \rangle$ be defined as follows
$$\phi(x):=(x_{|1}, x_{|2}, \dots)$$
Then $\phi$ is topological isomorphism.
\end{thm}
\begin{proof}
Firstly, we will show that $\phi$ is one-to-one. Take $x=(x_0, x_1, \dots), y= (y_0, y_1, \dots) \in \mathbb{P}$ and assume $\phi(x)=\phi(y)$. Then $x_{|n}=y_{|n}$ for each $n \in \N$, which means that $x=y$ and consequently $\phi$ is one-to-one.

Now we show that $\phi$ is onto. Let $\mathbf{a}=(a_1,a_2,\dots) \in \varprojlim \langle \{P_n\}_{n \in \N},\{p^n_k \}_{k < n} \rangle$. Then $p_n(a_{n+1})=a_{n+1}|_n=a_n$. Let $x_n=a_{n+1}(n)$ and $x=(x_0,x_1,\dots)$. Then $\mathbf{a}=(x_{|1}, x_{|2}, \dots)$ and $\phi(x)=\mathbf{a}$.

Now we will show that $\phi$ is continuous. Take any basic set \[(\{x_{|1}\} \times \{x_{|2}\} \times \dots \times \{x_{|n+1}\} \times P_{n+1} \times \dots) \cap  \varprojlim \langle  \{P_n\}_{n \in \N},\{p^n_k \}_{k < n} \rangle.\] Then \[\phi^{-1}((\{x_{|1}\} \times \{x_{|2}\} \times \dots \times \{x_{|n+1}\} \times P_{n+1} \times \dots) \cap  \varprojlim \langle  \{P_n\}_{n \in \N},\{p^n_k \}_{k < n} \rangle)=\] \[\{x_0\} \times \{x_1\} \times \dots \times \{x_n\} \times \{0, 1, 2, 3\} \times \dots,\] which means $\phi$ is continuous. Since $\mathbb{P}$ is compact, $\phi$ is a homeomorphism. By Lemma \ref{OrderingsAreTheSameOnPandOnInverseLimit} $\phi$ is a topological isomorphism.
\end{proof}
Our next step is to show that sequence \eqref{SequanceOfUniwersalProjectivePosets} is a \text{\Fraisse} sequence.
\begin{prop}
A sequence \eqref{SequanceOfUniwersalProjectivePosets} is a \text{\Fraisse} sequence in the category of finite posets with quotient maps.
\end{prop}
\begin{proof}
From Lemma \ref{APropertyImpliesUProperty} we get that, if our sequence has \hyperlink{AProperty}{(A)} property, then it also has \hyperlink{UProperty}{(U)} property as $\K$ is directed. So we only need to show that our sequence has \hyperlink{AProperty}{(A)} property.

Let $H$ be a finite poset and let $p:H \rightarrow P_k$ be a quotient map. Note that for every $x \in P_k$ there is only one $y \in P_k$ such that $x \leq_k y$ and $x \neq y$. Since $P_n$ is isomorphic to $\dot\bigcup_{i < 2 \cdot 4^{n-1} } \mathbf{2}$ for all $n \in \N$, then the cardinality of $\{(x,y)\in P_k\times P_k:x<_ky\}$ equals $2\cdot 4^{k-1}$. Let $D_i:= \lbrace x_i, y_i \rbrace$ be such that $x_i <_k y_i$ for $i=1,\dots,2 \cdot 4^{k-1}$. Let $m>k$. Then for every $(x,y)$ with $x<_ky$ we have the following:
\begin{enumerate}
\item We have at least $2^{m-k-1}$ pairs $(u,v)$ with  $u <_m v$ and such that $p_k ^{m}(u)=p_k ^{m}(v)=x$. Indeed if we take pairs $(u,v)$ of the form $u=x\hat{\;\;}2\hat{\;\;}t$ and $v=x\hat{\;\;}3\hat{\;\;}t$ where $t$ is 0-1 sequence of the length $m-k-1$, we get that $p_k ^{m}(u)=p_k ^{m}(v)=x$ and there is  $2^{m-k-1}$ of such pairs. The same applies if we substitute $x$ with $y$.
\item We have $2^{m-k}$ pairs $(u,v)$ with  $u <_mv$ and such that $p_k ^{m}(u)=x,p_k ^{m}(v)=y$; they are of the form $u=x\hat{\;\;}t$ and $v=y\hat{\;\;}t$ where $t$ is 0-1 sequence of the length $m-k$.
\end{enumerate}
Additionally we will show that $(p_k^m)^{-1}(\{x,y\})$ does not contain isolated points. Let $v\in (p_k^m)^{-1}(\{x,y\})$. Since $P_m$ does not contain isolated points, there is $w$ with $w<_m v$ or $v<_m w$. Since $p_k^m$ is order-preserving, $p_k^m(w)$ is comparable with $p_k^m(v)\in\{x,y\}$. Thus $p_k^m(w)\in\{x,y\}$, and therefore $v$ is not isolated in $(p_k^m)^{-1}(\{x,y\})$.    

Now we fix $m$ such that $2^{m-k-1}$ is greater than the cardinality of $\{(x,y):x \leq_H y \}$. For every $i \leq 2 \cdot 4^{k-1}$ we will find a quotient $g_i :(p_k ^{m})^{-1} \left(D_i\right) \rightarrow p^{-1} \left(D_i\right)$ such that diagram \begin{equation}\label{g_i diagram}
\begin{tikzcd}
D_i && (p_k ^{m})^{-1} \left(D_i\right) \arrow[ll, "p_k ^{m}"'] \arrow[ld,dashed, "g_i"] \\
	& p^{-1} \left(D_i\right) \arrow[lu, "p"]
\end{tikzcd}
\end{equation}
commutes.

Suppose that we have already found such $g_i$'s. We define $g:P_m \rightarrow H$ as follows:
$g(x)=g_i(x)$ provided $x \in (p_k ^{m})^{-1} \left(D_i\right)$. Since $D_i \cap D_j = \emptyset$ for $i \neq j$, $g$ is well-defined. Let $y\in H$. Then $p(y) \in P_k=\bigcup_i D_i$, thus there is $i$ with $p(y) \in D_i$, which implies $y \in p^{-1}(D_i)$. Since  $g_i:(p_k ^m)^{-1}(D_i) \to p^{-1}(D_i)$ is a quotient map we can find $x \in (p_k ^m)^{-1}(D_i)$ with $g_i(x)=y$. Then $g(x)=g_i(x)=y$, thus $g$ is onto. 
Assume that $x\leq_m x'$. Then $p^m_k(x)\leq_k p^m_k(x')$ and $p^m_k(x), p^m_k(x')$ belong to the same $D_i$. Since $g_i$ is homomorphism, then $g(x)=g_i(x)\leq_H g_i(x')=g(x')$ and $g$ is homomorphism as well. Let $u,u'\in P_k$ with $u<_k u'$. There is $i$ with $D_i=\{u,u'\}$. Since $p$ is a quotient map, there are $y,y'\in p^{-1}(D_i)$ with $y<_H y'$ and $p(y)=u$, $p(y')=u'$. Since $g_i$ is a quotient map, there are $x<_m x'$ with $g_i(x)=u$ and $g_i(x')=u'$, and $g$ is a quotient map as well.  
By \eqref{g_i diagram} the diagram \begin{equation}
\begin{tikzcd}
P_k && P_m \arrow[ll, "p_k ^{m}"'] \arrow[ld, "g"] \\
	& H \arrow[lu, "p"] \\
\end{tikzcd}
\end{equation}
commutes.

Now it is enough to construct $g_i$'s. Let $D_i= \lbrace L, U \rbrace$ with $L<U$. Now we will categorize points and edges in $p^{-1}(D_i)\subseteq H$. Isolated points of $p^{-1}(D_i)$ we divide into L-isolated (denoted by $LI$) and U-isolated (denoted by $UI$) if $p$ maps them onto $L$ and $U$ respectively. Further, pairs $(u,v)$ with $u<_H v$ in $p^{-1}(D_i)$ are divided into $LP$ if $p(u)=p(v)=L$, $UP$ if $p(u)=p(v)=U$ and $LUP$ if $p(u)=L,p(v)=U$. Since $p$ is a quotient map then $LUP \neq \emptyset$. Let us fix $r,b \in p^{-1}(D_i)$ such that $(r,b) \in LUP$. 

Now we will do a similar thing for $(p_k ^{m})^{-1} \left(D_i\right)$. We have already noticed that it does not contain any isolated points. We divide its pairs $(x,y)$ with $x <_m y$ into three categories:
L-type $\mathbf{L}$ if $p_k ^{m}$ maps $x$ and $y$ into $L$;
U-type $\mathbf{U}$ if $p_k ^{m}$ maps $x$ and $y$ into $U$;
LU-type $\mathbf{LU}$ if $p_k ^{m}$ maps $x$ into $L$ and $y$ into $U$.

Let $\ell_1=|LI|,\ell_2=|UI|,\ell_3=|LP|,\ell_4=|UP|,\ell_5=|LUP|$. Then $\ell_1,\ell_2,\ell_3,\ell_4 \geq 0$ and $\ell_5 \geq 1$. By definition of $m$, $\ell_1$ and $\ell_3$ we obtain $|\mathbf{L}| \geq 2^{m-k-1} \geq |\{(x,y):x\leq_H y\}|$. Note that we can identify the set $LI$ with set of pairs $\{(x,x): x \in LI\}$ and this set is disjoint with $LUP$, and therefore  $|\{(x,y):x\leq_H y\}| \geq \ell_1 +\ell_3$. So we can find pairwise disjoint $\mathbf{L}_1, \mathbf{L}_2, \mathbf{L}_3$ such that $\mathbf{L}_1 \cup \mathbf{L}_2 \cup \mathbf{L}_3 = \mathbf{L}$ and $|\mathbf{L}_1|=\ell_1,|\mathbf{L}_2|=\ell_3$. Similarly we divide $\mathbf{U}$ into pairwise disjoint sets $\mathbf{U}_1, \mathbf{U}_2, \mathbf{U}_3$ such that $\mathbf{U}_1 \cup \mathbf{U}_2 \cup \mathbf{U}_3 = \mathbf{U}$ and $|\mathbf{U}_1|=\ell_2,|\mathbf{U}_2|=\ell_4$. Now we partition $\mathbf{LU}$ into $\mathbf{LU}_1, \mathbf{LU}_2$ with $|\mathbf{LU}_1|=\ell_5$. Now we are ready to define $g_i$. Since $P_m$ is isomorphic to $\dot\bigcup_{i < 2 \cdot 4^{m-1} } \mathbf{2}$, then we will define $g_i$ for each pair $x,y \in P_m$ with $x<_m y$ as follows :
\begin{itemize}
\item[(1)] if $(x,y)$ is a $j$-th pair in $\mathbf{L}_1$, then $g_i(x)=g_i(y)$ is the $j$-th element of $LI$;
\item[(2)] if $(x,y)$ is a $j$-th pair in $\mathbf{L}_2$, then $(g_i(x),g_i(y))$ is the $j$-th pair of $LP$;
\item[(3)] if $(x,y)$ is in $\mathbf{L}_3$, then $g_i(x)=g_i(y)=r$ (recall that $r,b \in p^{-1}(D_i)$ have been fixed such that $(r,b) \in LUP$);
\item[(4)] if $(x,y)$ is a $j$-th pair in $\mathbf{U}_1$, then $g_i(x)=g_i(y)$ is the $j$-th element of $UI$;
\item[(5)] if $(x,y)$ is a $j$-th pair in $\mathbf{U}_2$, then $(g_i(x),g_i(y))$ is the $j$-th pair of $UP$;
\item[(6)] if $(x,y)$ is a $j$-th pair in $\mathbf{U}_3$, then $g_i(x)=g_i(y)=b$;
\item[(7)] if $(x,y)$ is a $j$-th pair in $\mathbf{LU}_1$, then $(g_i(x),g_i(y))$ is a $j$-th pair in $LUP$; 
\item[(8)] if $(x,y)$ is in $\mathbf{LU}_2$, then $(g_i(x),g_i(y))=(r,b)$.
\end{itemize}
By (1), (2), (4), (5) and (7) $g_i$ is onto $p^{-1}(D_i)$. Fix $(x,y)\in (p^m_k)^{-1}(D_i)\times (p^m_k)^{-1}(D_i)$ with $x <_m y$. If $p^m_k(x)=p^m_k(y)=L$, then $p(g_i(x))=p(g_i(y))=L$ by (1)--(3). If $p^m_k(x)=p^m_k(y)=U$, then $p(g_i(x))=p(g_i(y))=U$ by (4)--(6). If $p^m_k(x)=L$ and $p^m_k(y)=U$, then $p(g_i(x))=L$ and $p(g_i(y))=U$ by (7)--(8). This shows that diagram \eqref{g_i diagram} commutes. The fact that $g_i$ is a homomorphism follows easily from (1)-(8). Note that $g_i$ is quotient map by (2), (5) and (7). 
\end{proof}

\subsection{Properties of the inverse \text{\Fraisse} limit $\mathbb{P}$}
Now we shall proof that in fact $\mathbb{P}$ is universal, in other words every profinite poset is an image of some continuous quotient mapping from $\mathbb{P}$. To achieve this we need to proof the following lemma.
\begin{lem}\label{FiniteSequencesLemma}
Let $(G_i)_{i \in \w}$ be a sequence of non-empty, finite sets and $q_k ^ i:G_i \rightarrow G_k$ be functions such that $q^k_\ell\circ q^i_k=q^i_\ell$ for every $\ell<k<i$. We define \[T_i := \lbrace (a_0, a_1, \ldots, a_i) \in G_0 \times G_1 \times \ldots \times G_i:q_k ^i(a_i)=a_k \text{ for every } k < i \rbrace\] for $i \in \w$. If $T_i$ is non-empty for every $i \in \w$ then there exists a sequence $(x_n) \in \prod_{i \in \w} G_i$ such that $q_n ^{n+1}(x_{n+1})=x_n$ for all $n \in \w$.
\end{lem}
\begin{proof}
Let us define $p_k ^i(a_0, a_1, \ldots, a_i)=(a_0, a_1, \ldots, a_k)$ for $k  \leq i$. Note that the condition $q^i_k(a_i)=a_k$ for every $k<i$ implies $q^k_\ell(a_k)=a_\ell$ for every $\ell<k<i$. 
Therefore $p_k ^i \left[T_i \right] \subseteq T_k$ for all $i \in \w$ and $k < i$. 
Thus $p_0 ^{i+1} \left[T_{i+1} \right]= p_0 ^{i} \left[p_i ^{i+1}\left[T_{i+1} \right]\right] \subseteq p_0 ^i \left[ T_i \right]$ for every $i$. Furthermore as $T_i$ are finite non-empty sets then $p_0 ^i \left[ T_i \right]$ are finite and not-empty. Because of that $\left(p_0 ^i \left[ T_i \right]\right)_{i \in \N}$ is a decreasing sequence of non-empty finite sets. Eventually this sequence must stabilize so $\bigcap_{i \in \N} p_0 ^i \left[ T_i \right]=p_0 ^ \ell\left[T_\ell \right] \neq \emptyset$ for some $\ell \in \w$. Similarly $\bigcap_{i\in\N}p_k ^{i+k}\left[T_{i+k} \right]\neq\emptyset$ for all $k \in \w$. Let $V_k:=\bigcap_{i \in \N} p_k ^{i+k}\left[T_{i+k} \right]$ for $k \in \w$. Note that $p_k ^i \left[V_i \right]=V_k$ for every $k \leq i$. Now we are ready to define the desired sequence. 

Let $x_0 \in G_0 $ such that $x_0 \in V_0$. As $p_0 ^1 \left[V_1 \right]=V_0$ we find $x_1 \in G_1$ such that $(x_0, x_1) \in V_1$, thus $q_0 ^1(x_1)=x_0$. In general we can assume that we have $(x_0, x_1 \ldots, x_n)$ such that $q_k ^{k+1}(x_{k+1})=x_k$ for $k < n$. We can find $x_{n+1} \in G_{n+1}$ such that $(x_0, x_1, \ldots, x_{n+1}) \in V_{n+1}$ because $p_n ^{n+1} \left[V_{n+1} \right]=V_n$. Therefore $q_n ^{n+1}(x_{n+1})=x_n$. So we can find a sequence $(x_0, x_1, \ldots) \in \prod_{i \in \w} G_i$ such that $q_n ^{n+1}(x_{n+1})=x_n$ for every $n \in \w$.
\end{proof}

In sequel we will use the following simple observation. 
\begin{obs}\label{TwoCommutingDiagrams}
Note that if the following diagrams
\begin{equation*}
\begin{tikzcd}
A \arrow[d] & \arrow[l] \arrow[d] B && B \arrow[d] & \arrow[l] E \arrow[d]\\
C & \arrow[l] D && D & \arrow[l] F \\
\end{tikzcd}
\end{equation*}
commute, then diagram
\begin{equation*}
    \begin{tikzcd}
    A \arrow[d] & \arrow[l] E \arrow[d] \\
    C & \arrow[l] F\\
    \end{tikzcd}
\end{equation*}
commutes as well.
\end{obs}

The following Theorem follows from \cite[Theorem 5.2]{kubis2014}. However, the proof there is very short and sketchy, with general model-theoretical assumptions. Anyone who wants to really understand this theory needs to complete the reasoning. Since we did not find detailed argument anywhere, we decided to present it here. This can be adapted to general situation.    

\begin{thm}\label{UniversalPosetMapping}
Let $H=\varprojlim\langle\{H_n\}_{n \in \N},\{h^n_k\}_{k<n}\rangle$ and $\mathbb{P}=\varprojlim\langle\{P_n\}_{n \in \N},\{p^n_k\}_{k<n}\rangle$, where $(P_n)$ is a \text{\Fraisse} sequence. Then there is a continuous quotient map $f:\mathbb{P} \rightarrow H$.
\end{thm}
\begin{proof}
We will define such quotient map using the fact that $\mathbb{P}$ is an inverse limit of \text{\Fraisse} sequence $\langle \{P_n\}_{n\in\N},\{p^n_k\}_{k < n}\rangle$. Because of \hyperlink{UProperty}{(U)} property we shall find quotient $f_1:P_{i_1} \rightarrow H_1$. Now using \hyperlink{ExtensionProperty}{(E)} property we will find such $P_{i_2},i_1<i_2$ and $f_2:P_{i_2} \to H_2$ that the diagram
\begin{equation}
\begin{tikzcd}
P_{i_1} \arrow[d, "f_1"'] & P_{i_2} \arrow[l, "p_{i_1} ^{i_2}"'] \arrow[d,"f_2"'] \\
H_1 					 & H_2 \arrow[l,"h_1 ^2"] \\
\end{tikzcd}
\end{equation}
commutes. Now we can apply \hyperlink{ExtensionProperty}{(E)} property to $f_2$ and $h_2 ^3$ to get next commuting diagram. Proceeding inductively we can apply this to $f_n$ and $h_n ^{n+1}$ to acquire respective commuting diagrams. Using Observation \ref{TwoCommutingDiagrams} we get that for every $n,k \in \N,k<n$ the diagram
\begin{equation} \label{CommutingDiagram}
\begin{tikzcd}
P_{i_k} \arrow[d, "f_k"'] & P_{i_n} \arrow[l, "p_{i_k} ^{i_n}"'] \arrow[d,"f_n"'] \\
H_k 					 & H_n \arrow[l,"h_k ^n"] \\
\end{tikzcd}
\end{equation}
commutes. Let $f:\mathbb{P} \to H$ be such that $f(x_1, x_2, \ldots)=(f_1(x_{i_1}), f_2(x_{i_2}), \ldots)$. We shall prove that it is a continuous quotient mapping onto $H$. It is easy to see that for every $n \in \N$ we have $h_n ^{n+1}(f_{n+1}(x_{i_{n+1}}))=f_n(p_{i_n} ^{i_{n+1}}(x_{i_{n+1}}))=f_n(x_{i_n})$. So $f(x_1, x_2, \ldots) \in H$ for every $(x_1, x_2, \ldots) \in \mathbb{P}$. We now show that $f$ is a homomorphism. Take $(x_1, \ldots) \leq_{\mathbb{P}} (y_1, \ldots)$. Then for every $n \in \N$ we have that $x_{i_n} \leq_{P_{i_n}} y_{i_n}$ so $f_n(x_{i_n}) \leq_{H_n} f_n(y_{i_n})$. Thus $f(x)=(f_1(x_{i_1}), \ldots) \leq_H  (f_1(y_{i_1}), \ldots)=f(y)$. 

Now we will show that $f$ is onto $H$. Take $(z_1, z_2, \ldots)$ in $H$. It is easy to see that $(f_n ^{-1} \left[ \lbrace z_n \rbrace \right])_{n \in \N}$ is a sequence of non-empty finite sets as $P_{i_n}$ are finite and $f_n$ are quotient maps. Let $M:= \prod_{n \in \N} f_n ^{-1} \left[\lbrace z_{n} \rbrace \right]$ and $$T_l:= \lbrace (x_{i_1}, x_{i_2}, \ldots, x_{i_l}) \in f_1 ^{-1} \left[\lbrace z_1 \rbrace \right] \times f_2 ^{-1} \left[\lbrace z_2 \rbrace \right] \times \ldots \times f_l ^{-1} \left[\lbrace z_l \rbrace \right]:p_{i_k} ^{i_l}(x_{i_l})=x_{i_k} \text{ for } k<l \rbrace.$$ We will show that $T_l$ is non-empty for every $l >1$. Take $x_{i_l} \in f_l ^{-1} \left[\lbrace z_l \rbrace \right]$. Then $z_{l-1}=h_{l-1} ^l(f_l(x_{i_l}))=f_{l-1}(p_{i_{l-1}} ^ {i_l}(x_{i_l}))$. Thus $x_{i_{l-1}}:=p_{i_{l-1}} ^ {i_l}(x_{i_l}) \in f_{l-1} ^{-1} \left[\lbrace z_{l-1} \rbrace \right]$. As diagram \eqref{CommutingDiagram} commutes we can continue this process to get a finite sequence in $T_l$ which makes it non-empty. This also shows that $p_{i_k} ^{i_n}$ maps $f_n ^{-1} \left[ \lbrace z_n \rbrace \right]$ to $f_k ^{-1} \left[ \lbrace z_k \rbrace \right]$ for $k<n$.
Furthermore $p_{i_k} ^{i_n} \circ p_{i_n} ^{i_l}=p_{i_k} ^{i_l}$ for $k<n<l$. Because of that we can use Lemma \ref{FiniteSequencesLemma} to get a sequence $(x_{i_1}, x_{i_2}, \ldots) \in M$ such that $p_{i_n} ^{i_{n+1}}(x_{i_{n+1}})=x_{i_n}$ for every $n \in \N$. Put $x'_n := p_n ^{i_n}(x_{i_n})$. Note that  $x'_{i_n}=x_{i_n}$. For that reason we will write $x_n$ instead of $x_n'$, which will not lead to any confusion. Then $f(x_1, x_2, \ldots)=(f_1(x_{i_1}), f_2(x_{i_2}), \ldots)=(z_1, z_2, \ldots)$ and $p_n ^{n+1}(x_{n+1})=p_n ^{n+1}(p_{n+1} ^{i_{n+1}}(x_{i_{n+1}}))=p_n ^{i_n} (p_{i_n} ^{i_{n+1}}(x_{i_{n+1}}))=p_n ^{i_n}(x_{i_n})=x_n$ so $(x_1, x_2, \ldots) \in \mathbb{P}$.

Now we show that $f$ is a strict homomorphism. Let $(u_1, u_2, \ldots),(v_1, v_2, \ldots) \in H$ such that $(u_1, u_2, \ldots) \leq_H (v_1, v_2, \ldots)$. Then $u_n \leq_{H_n} v_n$ for every $n \in \N$. Because $f_n$ is a quotient map the sets $A_n:= \lbrace (x,y) \in P_{i_n} \times P_{i_n}: x \leq_{P_{i_n}} y, f_n(x)=u_n, f_n(y)=v_n \rbrace$ are non-empty. Now for $n \in \N$, $(x_{i_n},y_{i_n}) \in A_n$ and $k < n$ we define $r_k ^n(x_{i_n}, y_{i_n}):=(p_{i_k} ^{i_n}(x_{i_n}), p_{i_k} ^{i_n}(y_{i_n}))$. See that $(A_i)_{i \in \N}$ is a sequence of non-empty, finite sets and $r_k ^n \circ r_n ^l=r_k ^l$ for $k<n<l$. We define $$E_n:= \lbrace ((x_{i_1}, y_{i_1}), (x_{i_2}, y_{i_2}), \ldots, ((x_{i_n}, y_{i_n})) \in A_1 \times A_2 \times \ldots \times A_n:r_k ^n(x_{i_n},y_{i_n})=(x_{i_k}, y_{i_k}) \rbrace.$$ We will show that $\rng \left(r_k ^n \right) \subset A_k$ for any $n \in \N$ and $k < n$. Take any $n \in \N, k < n$ and $(x_{i_n}, y_{i_n}) \in A_n$. Because the diagrams \eqref{CommutingDiagram} commute we get that $f_k(p_{i_k} ^{i_n}(x_{i_n}))=u_k, f_k(p_{i_k} ^{i_n}(y_{i_n}))=v_k$ and $p_{i_k} ^{i_n}(x_{i_n}) \leq_{P_{i_k}} p_{i_k} ^{i_n}(y_{i_n})$ so $r_k ^n(x_{i_n}, y_{i_n}) \in A_k$ for any $n \in \N$ and $k < n$, thus $r_k ^n$ maps $A_n$ to $A_k$ for $k<n, n \in \N$. Furthermore this shows that $E_n$ is non-empty for every $n \in \N$. 

Thus we can again apply Lemma \ref{FiniteSequencesLemma} to get a sequence $(x_{i_n}, y_{i_n})_{n\in\N} \in \prod_{n \in \N} A_n$ such that $r_n ^{n+1}(x_{i_{n+1}}, y_{i_{n+1}})=(x_{i_{n}}, y_{i_{n}})$. So we have got $(x_{i_n}), (y_{i_n})$ such that $$p_{i_n} ^{i_{n+1}}(x_{i_{n+1}})=x_{i_n}, p_{i_n} ^{i_{n+1}}(y_{i_{n+1}})=y_{i_n}, x_{i_n} \leq_{P_{i_n}} y_{i_n} \text{ and } f_n(x_{i_n})=u_n, f_n(y_{i_n})=v_n.$$ 

Now take $x'_n:=p_n ^{i_n}(x_{i_n}), y'_n:=p_n ^{i_n}(y_{i_n})$. Note that $x'_{i_n}=x_{i_n}$ and $y'_{i_n}=y_{i_n}$. For that reason we will write $x_n$ and $y_n$ instead of $x'_n$ and $y'_n$, which will not lead to any confusion. Then $$f(x_1, x_2, \ldots)=(f_1(x_{i_1}), f_2(x_{i_2}), \ldots)=(u_1, u_2, \ldots),$$ 
$$f(y_1, y_2, \ldots)=(f_1(y_{i_1}), f_2(y_{i_2}), \ldots)=(v_1, v_2, \ldots), (x_1 ,x_2 \ldots) \leq_\mathbb{P} (y_1, y_2, \ldots).$$ 

Furthermore for every $n \in \N$, $p_n ^{n+1}(x_{n+1})=p_n ^ {n+1}(p_{n+1} ^ {i_{n+1}}(x_{i_{n+1}})=p_n ^{i_n}(p_{i_n} ^{i_{n+1}}(x_{i_{n+1}})=p_n ^{i_n}(x_{i_n})=x_n$. The same applies to $y_n$. Thus $f$ is a strict homomorphism.

Now we will show that $f$ is continuous. Take any basic set $U_{(z_1, z_2, \ldots, z_n)}=(\{z_1\} \times \{z_2\} \times \ldots \times \{z_n\} \times H_{n+1} \times \ldots) \cap H$. Let $(x_1, x_2, \dots) \in f^{-1}\left[U_{(z_1, z_2, \ldots, z_n)}\right]$. Then $f_1(x_{i_1})=z_1, \dots, f_n(x_{i_n})=z_n$. See that $(x_1, x_2, \dots) \in (\{x_1\} \times \{x_2\} \times \dots \times \{x_{i_1}\} \times \dots \times \{x_{i_n}\} \times P_{i_n +1} \times \dots) \cap \mathbb{P} \subset f^{-1}\left[U_{(z_1, z_2, \ldots, z_n)}\right]$ which means $f^{-1}\left[U_{(z_1, z_2, \ldots, z_n)}\right]$ is open, thus $f$ is continuous.
\end{proof}
\begin{rem}
The representation of a function $f$ in Theorem \ref{UniversalPosetMapping} is worth remembering, so we remind it for reader's convenience. 
Let $H$ a profinite poset generated by a sequence $H_n$ and quotient mappings $h_k ^n:H_n \to H_k$. 
There is a sequence $i_1<i_2<\dots$ and quotient mappings $f_n:P_{i_n} \to H_n$ such that the diagram
\begin{equation*}
\begin{tikzcd}
P_{i_n} \arrow[d, "f_n"'] & P_{i_k} \arrow[l, "p_{i_n} ^{i_k}"'] \arrow[d,"f_k"'] \\
H_n 					 & H_k \arrow[l,"h_n ^k"] \\
\end{tikzcd}
\end{equation*}
commutes for every $n,k \in \N,n<k$. Then we define a continuous quotient mapping $f:\mathbb{P} \to H$ as follows $f(x_1, x_2, \dots)=(f_1(x_{i_1}), f_2(x_{i_2}), \dots)$.
\end{rem}
Now we will show that similarly all continuous quotient maps $f:\mathbb{P} \to H$ where $H$ is a profinite poset are of this form.
\begin{lem}\label{Lemma_About_Representation_Of_Quotient_Mappings_From_P}
Let $f:P \to L$ be a continuous quotient mapping where $L$ is a finite poset and $P=\varprojlim\langle \{P_n\}_{n \in \N},\{p^n_k\}_{k<n}\rangle$ is a profinite poset. 
Let $i_0\in\N$. Then we can find $i>i_0$ and quotient mapping $h:P_i \to L$ such that $f=h \circ p_i$ where $p_i:P \to P_i$ is defined as follows: $p_i(x_1, x_2, \dots)=x_i$.
\end{lem}
\begin{proof}
We remind that the non-empty basic sets in $P$ have the following representation:
$$(\{x_1\} \times \{x_2\} \times \dots \times \{x_k\} \times P_{k+1} \times P_{k+2} \dots) \cap P $$
where $p_j ^n(x_n)=x_j$ for $j<n\leq k$.

We will say that this basic set has a rank $k$. As $P_k$ are all finite then every basic set of rank $k$ is a finite disjoint union of some basic sets of rank $k+1$.	

Let $L=\{l_1, l_2, \dots, l_n \}$. As $L$ has a discrete topology, we get that $f^{-1}[\{l_j\}]$ is a compact set as it is a closed subset of a compact space. Take any open cover of $f^{-1}[\{l_j\}]$ consisting of basic sets and take its finite sub-cover $\mathcal{V}_j$. Since $f^{-1}[\{l_j\}]$ is open, we may assume that $\bigcup \mathcal{V}_j= f^{-1}[\{l_j\}]$. Without the loss of generality we can assume that all such sub-covers consist of sets with the same rank greater than $i_0$, say $i$. From now on the rank $i$ is fixed.

Let 
\begin{equation}\label{EQ1}
(\{x_1\} \times \{x_2\} \times \dots \times \{x_i\} \times P_{i+1} \times P_{i+2} \dots) \cap P
\end{equation}
be a basic set of rank $i$. Clearly for all non-empty sets $(\{x'_1 \} \times \{x'_2 \} \times \dots \times \{x'_{i-1} \} \times \{x_i \} \times P_{i+1} \times P_{i+2} \dots) \cap P $ we have that $x'_j =x_j$ for $j < i$ as $p_j ^i(x_i)=x_j$. So each set of the form \eqref{EQ1} is fully described by $x_i$.

We define $h:P_i \to L$ as follows: $h(x_i)=l$ if and only if $(\{p_1 ^i(x_i) \} \times \{p_2 ^i(x_i) \} \times \dots \times \{x_i \} \times P_{i+1} \times P_{i+2} \times \dots) \cap P \subset f^{-1}[\{l\}]$. By the definition of $i$, each non-empty basic set of rank $i$ is contained in exactly one fiber $f^{-1}[\{l\}]$, which shows that $h$ is well-defined. 

Now we prove that $f=h \circ p_i$. Take any $(x_1, x_2, x_3, \dots) \in P$. Then $f(x_1, \dots, x_i, x_{i+1}, \dots) =l_j$ for some $l_j \in L$. 
Clearly $(\{x_1\} \times \{x_2\} \times \dots \times \{x_i\} \times P_{i+1} \times \dots) \cap P \subset f^{-1}[\{l_j\}]$ then and we have that $h(x_i)=l_j$. So we have that $h(p_i(x_1,\dots, x_i, x_{i+1}, \dots))=f(x_1,\dots, x_i, x_{i+1}, \dots)$

Now we prove that $h$ is onto $L$. Take any $l_j \in L$. We can find a basic set $W_{x_i}=(\{x_1\} \times \{x_2\} \times \{x_3\} \times \dots \times \{x_i\} \times P_{i+1} \times \dots) \cap P$ in $\mathcal{V}_j$. Then $h(x_i)=l_j$.

Now we prove that $h$ is a homomorphism.
Take any $x_i, y_i \in P_i$ such that $x_i \leq_{P_i} y_i$. Define $x_j:=p_j ^i(x_i), y_j:=p_j ^i(y_i)$ for $j < i$. As $p_j ^i$ are homomorphisms we get that $x_j \leq_{P_j} y_j$. As $p_i ^{i+1}$ is a quotient map we can find $x_{i+1}, y_{i+1} \in P_{i+1}$ such that $p_i ^{i+1}(x_{i+1})=x_i, p_i ^{i+1}(y_{i+1})=y_i$ and $x_{i+1} \leq_{P_{i+1}} y_{i+1}$. Proceeding inductively we get $$(x_1, x_2, \dots, x_i, x_{i+1}, \dots), (y_1, y_2, \dots, y_i, y_{i+1}, \dots) \in P$$ such that $$(x_1, x_2, \dots, x_i, x_{i+1}, \dots) \leq_{P} (y_1, y_2, \dots, y_i, y_{i+1}, \dots).$$ So $h(x_i)=f(x_1, x_2, \dots, x_i, x_{i+1}, \dots) \leq_{L} f(y_1, y_2, \dots, y_i, y_{i+1}, \dots)=h(y_i)$ thus $h$ is a homomorphism.

Now we show that $h$ is a strict homomorphism. Let $l,z \in L$ be such that $l \leq_L z$. As $f$ is a quotient map we get that for some $(x_1, x_2, \dots, x_i, \dots), (y_1, y_2, \dots, y_i, \dots) \in P$ such that $$(x_1, x_2, \dots, x_i, \dots) \leq_P (y_1, y_2, \dots, y_i, \dots)$$ and $$f(x_1, x_2, \dots, x_i, \dots)=l,f(y_1, y_2, \dots, y_i, \dots)=z.$$ Then $x_i \leq_{P_i} y_i$ and $h(x_i)=l, h(y_i)=z$.

\end{proof}

Now we may prove the following:
\begin{thm}\label{RepresentationOfContinuousQuotientMapping}
Let $P=\varprojlim\langle \{P_n\}_{n \in \N},\{p^n_k\}_{k<n}\rangle$,  $H=\varprojlim\langle\{H_n\}_{n \in \N},\{h^n_k\}_{k<n}\rangle$ be profinite posets and let $f:P \to H$ be a continuous quotient mapping. Then there is a strictly increasing sequence $(i_n)_{n \in \N}$ and quotient maps $g_j:P_{i_j} \to H_j$ such that $f(x_1, x_2, \dots)=(g_1(x_{i_1}), g_2(x_{i_2}), \dots)$. Furthermore the diagram
\begin{equation}\label{RepresentationDiagram}
    \begin{tikzcd}
    P_{i_k} \arrow[d, "g_k"'] & P_{i_n} \arrow[l, "p_{i_k} ^{i_n}"'] \arrow[d,"g_n"'] \\
    H_k 					 & H_n \arrow[l,"h_k ^n"] \\
    \end{tikzcd}
\end{equation}
commutes for every $n,k \in \N$ with $k<n$.
\end{thm}
\begin{proof}
Take any $(x_1, x_2, \dots) \in P$. Let $h_k:H \to H_k$ be the projection on $k$-th coordinate, i.e. $h_k(y_1, y_2, \dots)=y_k$. Since $f$ and $h_k$ are onto, then $\rng(h_k \circ f)=H_k$. So we can apply Lemma \ref{Lemma_About_Representation_Of_Quotient_Mappings_From_P} to continuous quotient map $h_k \circ f$. Using Lemma \ref{Lemma_About_Representation_Of_Quotient_Mappings_From_P} inductively we find a strictly increasing sequence $(i_k)$ of indices and  quotient mappings $g_k:P_{i_k} \to H_k$ such that $g_k \circ p_{i_k}=h_k \circ f$ where $p_{i_k}(x_1, x_2, \dots)=x_{i_k}$. So we get that $h_k(f(x_1, x_2, \dots))=g_k(x_{i_k})$. Thus the $k$-th coordinate of $f(x_1, x_2, \dots)$ is $g_k(x_{i_k})$. 

Now we are going to show that diagram \eqref{RepresentationDiagram} commutes. Take any $n,k \in \N$ with $k<n$ and let $x_{i_n} \in P_{i_n}$. Let $x_s:=p_s ^{i_n}(x_{i_n})$ for $s<i_n$. Then $(x_1, x_2, \dots, x_{i_n -1}, x_{i_n}, z_{i_n +1}, \dots) \in P$ for some $z_{i_n + \ell} \in P_{i_n + \ell}$ with $\ell \in \N$.
Then $$f(x_1, x_2, \dots, x_{i_n -1}, x_{i_n}, z_{i_n +1}, \dots)=(g_1(x_{i_1}), g_2(x_{i_2}), \dots, g_k(x_{i_k}), \dots, g_{n}(x_{i_n}), \dots) \in H,$$ so $h_k ^n(g_n(x_{i_n}))=g_k(x_{i_k})=g_k(p_{i_k} ^{i_n}(x_{i_n})$, which means $g_k \circ p_{i_k} ^{i_n}=h_k ^n \circ g_n$, thus diagram \eqref{RepresentationDiagram} commutes.
\end{proof}

Now we are going to prove that if we have a continuous quotient map $f$ from $\mathbb{P}$ onto a finite poset $A$ and a quotient mapping $g$ from finite poset $B$ onto $A$, then we can find a continuous quotient map $h$ from $\mathbb{P}$ onto $B$, such that $f=g \circ h$.

\begin{thm}\label{UniversalProperty}
Let $f:\mathbb{P} \to A, g:B \to A$ be continuous quotient mappings, where $A, B$ are finite posets and $\mathbb{P}=\varprojlim\langle\{P_n\}_{n \in \N},\{p^n_k\}_{k<n}\rangle$ where $(P_n)$ is a \text{\Fraisse} sequence. Then there exists a continuous quotient mapping $h:\mathbb{P} \to B$, such that diagram
\begin{equation*}
		\begin{tikzcd}
		A && \mathbb{P} {\arrow[ld, "h"]} \arrow[ll, "f"'] \\
		& B \arrow[lu, "g"] & \\
		\end{tikzcd}
\end{equation*}
commutes.
\end{thm}
\begin{proof}
From Lemma \ref{Lemma_About_Representation_Of_Quotient_Mappings_From_P} we get that there is there is a continuous quotient map $t: P_{i_1} \to A$ such that diagram 
\begin{equation}    
     \begin{tikzcd}
		A && \mathbb{P} \arrow[ld, "p_{i_1}"] \arrow[ll, "f"'] \\
		& P_{i_1} \arrow[lu, "t"] & \\
		\end{tikzcd}
\end{equation}
commutes, where $p_{i_1}(x_1, x_2, \dots)=x_{i_1}$.

Using \hyperlink{ExtensionProperty}{(E)} property for $A,B,P_{i_1}$, we shall find $i_2 > i_1$ and quotient map $l:P_{i_2} \to P_{i_1}$, such that diagram
\begin{equation}
    \begin{tikzcd}
    P_{i_1} \arrow[d, "t"'] &   \arrow[l, "p_{i_1} ^{i_2}"']  \arrow[d, "l"] P_{i_2} \\
    A & B \arrow[l, "g"]  \\
    \end{tikzcd}
\end{equation}
commutes. Note that $l$ is continuous. Let $p_{i_2}(x_1, x_2, \dots)=x_{i_2}$
Define $h:=l \circ p_{i_2}$. Clearly $h$ is continuous and $g \circ h= g \circ l \circ p_{i_2}=t \circ p_{i_1} ^{i_2} \circ p_{i_2}=t \circ p_{i_1}=f$, thus we get the assertion.
\end{proof}

Now we are going to show that $\mathbb{P}=\{0, 1, 2, 3\}^\w$ has a dense $G_{\delta}$ set of isolated points with respect to its ordering.

\begin{thm}\label{IsolatedPointsAreG-Delta}
Let $(X,d)$ be a compact metric space with closed relation $E$. Then the set of $E$-isolated points is $G_{\delta}$ subset of $X$. 
\end{thm}
\begin{proof}
Let $I$ be the set of all $E$-isolated points in $X$. Take any $x \in I$ and $\epsilon >0$. Then the set $Y_{x,\epsilon}:=\{y \in X:d(x,y) \geq \epsilon \}$ is closed. Now, we will prove that for each $y \in Y_{x,\epsilon}$, there are its open neighborhood $U_y$ and $\delta_y>0$ such that
\begin{equation}\label{NotBeingInRelation}
\text{ for every } x', y' \in X \text{ if } d(x, x')<\delta_y \text{ and } y' \in U_y \text{, then } (x', y') \notin E \text{ and } (y', x') \notin E \tag{$\spadesuit$} 
\end{equation}
Suppose to the contrary that there is $y \in Y_{x,\epsilon}$, such that for every open neighborhood $U_y$ and $\delta_y>0$ there are $x',y' \in X$, such that $d(x, x')< \delta_y$ and $y' \in U_y$ and $(x', y') \in E$ or $(y', x') \in E$. We can find sequences $x_n', y_n'$, such that $d(x, x_n')<\frac{1}{n+1}$ and $y_n' \in B(y, \frac{1}{n+1})$. Then $x_n'  \to x, y_n' \to y$ and either $(x_n', y_n' ) \in E$ or $(y_n ', x_n') \in E$ for each $n \in \w$. One of these conditions, either $(x_n', y_n' ) \in E$ or $(y_n ', x_n') \in E$, is satisfied for infinitely many $n$'s, so we can choose sub-sequences $(x' _{k_n}), (y' _{k_n})$, which satisfy one of them for every $n$. Assume that $ (x' _{k_n}, y' _{k_n}) \in E$ for each $n \in \w$. Then as $E$ is closed we get that $(x,y) \in E$, which is a contradiction. If $(y' _{k_n}, x' _{k_n}) \in E$ for each $n$ we reach a contradiction as well.

Since $Y_{x,\epsilon}$ is compact, there are open neighborhoods $U_{y_1}, \dots, U_{y_n}$ of $y_1, \dots, y_n$, respectively, such that $Y_{x,\epsilon} \subset \bigcup_{i=1} ^n U_{y_i}$. Let $\delta(x, \epsilon):= \min\{\epsilon, \delta_{y_1}, \delta_{y_2}, \dots, \delta_{y_n} \}$. See that, if $x'$ has a distance less that $\delta(x, \epsilon)$ from $x$ and $y \in Y_{x, \epsilon}$, then $(x', y) \notin E$ and $(y,x') \notin E$. For each $\epsilon>0$ define
$$V_{\epsilon}:=\{x' \in X: d(x, x') <\delta(x, \epsilon)\text{ for some }x\in I \}.$$

Then $V_\epsilon$ is open. We will show that $$I= \bigcap_{n \geq 1} V_{\frac{1}{n}}.$$ It is clear that $I \subset \bigcap_{n \geq 1} V_{\frac{1}{n}}$. On the other hand, if $y \in X \setminus I$, then $(y,z) \in E$ or $(z,y) \in E$ for some $z \in X$. Let $n \in \N$ be such that $\frac{2}{n}<d(z,y)$. We will show that $y \notin V_{\frac{1}{n}}$. Suppose to the contrary that $y \in V_{\frac{1}{n}}$. Let $x \in I$ be such that $d(x,y)<\delta(x, \frac{1}{n})\leq \frac{1}{n}$. See that $\frac{2}{n}<d(z,y) \leq d(z,x) + d(x,y)<d(z,x)+\frac{1}{n}$, so $d(z,x)>\frac{1}{n}$, which means $z \in Y_{x, \frac{1}{n}}$. From condition \eqref{NotBeingInRelation} we get that $(y,z) \notin E$ and $(z,y) \notin E$, which gives us a contradiction.
This proves that $I=\bigcap_{n \in \N}V_{\frac{1}{n}}$, which means $I$ is a $G_\delta$ set.
\end{proof}
\begin{thm}
$\mathbb{P}=\{0,1,2,3\}^\w$ has a dense $G_\delta$ set of isolated points.
\end{thm}
\begin{proof}
From Lemma \ref{RelationOnInverseLimitIsClosed}, Theorem \ref{PIsTopIsoToInvFrLim} and  Theorem \ref{IsolatedPointsAreG-Delta} we get that the set of isolated points in $\mathbb{P}$ is $G_\delta$. So we only need to show that it is dense. 

Let $A:=\{(x_0, x_1, \dots) \in \mathbb{P}: x_n=2$ for all but finitely many $n$'s$\}$. Immediately from the definition of order relation on $\mathbb{P}$ we obtain that $A \subset I$, where $I$ is the set of isolated points in $\mathbb{P}$. We will show that $A$ is dense in $\mathbb{P}$. Take any basic set $\{x_0\} \times \{x_1\} \times \dots \times \{x_n\} \times \{0,1,2,3\} \times \dots$. See that $(x_0, x_1, x_2, \dots, x_n, 2, 2, 2, \dots)$ is an element of both $A$ and  $\{x_0\} \times \{x_1\} \times \dots \times \{x_n\} \times \{0,1,2,3\} \times \dots$, which means $A$ is dense, thus $I$ is also dense.
\end{proof}

\section{The inverse system of order ideals}\label{Section_The_inverse_system_of_order_ideals}
\subsection{Induced quotient maps}
Now we are going to study the following problem. Let $\langle\{P_n\}_{n \in \N},\{p^n_k\}_{k<n}\rangle$ be an inductive sequence consisting of finite posets with quotient maps and $P=\varprojlim \langle\{P_n\}_{n \in \N},\{p^n_k\}_{k<n}\rangle$. From Lemma \ref{Isomorphism_Of_P_And_PD} we know that each of $P_n$ embeds into its order ideal $\O(P_n)$ by mapping $x \mapsto \dol x$. The question is whether we can find quotient mapping from $\O(P_{n+1})$ onto $\O(P_n)$ such that the following diagrams 
\begin{equation}\label{DiagramToFill}
	\begin{tikzcd}
	P_1 \arrow[d, "\downarrow x"'] & \arrow[d, "\downarrow x"'] \arrow[l, "p_1 ^2"'] P_2 & \arrow[l, "p_2 ^3"'] \arrow[d, "\downarrow x"'] P_3 & \arrow[l, "p_3 ^4"'] \ldots & P \arrow[d, "?"]  \\
	\O(P_1) & \arrow[l, "?"] \O(P_2) & \arrow[l, "?"] \O(P_3) & \arrow[l, "?"] \ldots & \text{ ? }  \\
	\end{tikzcd}
\end{equation}
commute. Our aim is to show that question marks in the diagram can be superseded by appropriate arrows and a limit object. We will study the properties of that arrows and an object. Firstly, in the following Theorem we define arrows between $\O(P_{n+1})$ and $\O(P_n)$. 
\begin{thm}\label{Posets_and_O(P)}
	Let $P,Q$ be finite posets. Let $\varphi:P \to \O(P)$ and $\psi:Q\to\O(Q)$ be the embeddings given by $x\mapsto\dol x$. Let $p:Q \to P$ be a quotient map. Let $\hat{p}: \O(Q) \to \O(P)$ be defined as follows: 
	\begin{itemize}
		\item $\hat{p}(\emptyset)=\emptyset$ 
		
		\item $\hat{p}(\dol x)=\dol p(x)$ for every $x \in Q$ 
		\item For every non-empty, non-principal down-set $A \in \O(Q)$ we take its canonical decomposition $\{\dol a_i: i<k\}$ and we put $\hat{p}(A)=\bigcup_{i=1} ^k \hat{p}(\dol a_i)$ 
	\end{itemize}
	Then $\hat{p}$ is a quotient map for posets, which is join-preserving and preserves join-irreducible elements.

	Furthermore the diagram
	\begin{equation}\label{DiagramWithP0}
	\begin{tikzcd}[ampersand replacement=\&]
	P \arrow[d, "\varphi"'] \& Q \arrow[l, "p"'] \arrow[d, "\psi"] \\
	\O(P)					\& \O(Q) \arrow[l, "\hat{p}"] \\
	\end{tikzcd}
	\end{equation}
	commutes.
 We will say that quotient map $\hat{p}$ is induced by quotient map $p$.
\end{thm}

Let us stress the fact that one cannot obtain a quotient map for lattices in such situation. We present an example that illustrates this. We say that a mapping $p$ preserves meets, if $p(x\wedge y)=p(x)\wedge p(y)$ for all $x$ and $y$.
\begin{exa}\label{ExampleNotPreservingMeets}
Consider posets $Q:=\overline{\mathbf{1}}$ and $P:=\overline{\mathbf{2}}$ and quotient map $p:P \rightarrow Q$ such that $p \equiv 1$. Let $\hat{p}$ be a homomorphism induced by $p$. Then $\hat{p}(\dol 1 \cap \dol 2)=\hat{p}(\emptyset)=\emptyset$ and $\hat{p}(\dol 1)= \dol 1=\hat{p}(\dol 2)$, thus $\hat{p}(\dol 1) \cap \hat{p}(\dol 2)=\dol 1 \neq \emptyset$. Therefore $\hat{p}$ does not preserve meets.
\end{exa}
Now we are ready to begin the proof of Theorem \ref{Posets_and_O(P)}.
\begin{proof}
 By the definition of $\hat{p}$ the diagram \eqref{DiagramWithP0} commutes and preserves join-irreducible elements.
Now we will show that $\hat{p}$ preserves set unions. Firstly, we show that $\hat{p}(\dol x \cup \dol y)=\hat{p}(\dol x) \cup \hat{p}(\dol y)$. Let us consider two cases: 
\begin{enumerate}
	\item[(i)] If $ \dol x \parallel \dol y$ then from Lemma \ref{DownsetLemma2} and definition of $\hat{p}$ we get $\hat{p}(\dol x \cup \dol y)= \hat{p}(\dol x) \cup \hat{p}(\dol y)$
	\item[(ii)] If $ \dol x \subset \dol y$ then $x\leq y$ which implies $p(x)\leq p(y)$. Thus $\hat{p}(\dol x)=\dol p(x) \subset  \dol p(y)=\hat{p}(\dol y)$ so $\hat{p}(\dol x) \cup \hat{p}(\dol y)=\hat{p}(\dol y)=\hat{p}(\dol x \cup \dol y)$.
	(Case $ \dol y \subset \dol x$ goes similarly).
\end{enumerate}
By induction we obtain $\hat{p}\left(\bigcup_{k=1} ^{n} \dol x_k\right)=\bigcup_{k=1} ^{n} \hat{p}(\dol x_k)$ for any $\dol x_1, \dots, \dol x_n \in \J(\O(Q))$. 
Now we deal with a general case. Let $X,Y \in \O(Q)$. Then $X= \bigcup_{k=1} ^{n} \dol x_k$ and $Y= \bigcup_{k=1} ^{m}  \dol y_k$. So $X \cup Y = \bigcup_{k=1} ^{n}  \dol x_k \cup \bigcup_{k=1} ^{m} \dol y_k$,  thus $\hat{p} (X \cup Y)= \bigcup_{k=1} ^{n}  \hat{p}(\dol x_k) \cup \bigcup_{k=1} ^{m} \hat{p}(\dol y_k)=\hat{p}(X) \cup \hat{p}(Y)$. Note that if $X \subset Y$, then $X \cup Y=Y$ and consequently $\hat{p}(X) \subset\hat{p}(X \cup Y)= \hat{p}(X) \cup \hat{p}(Y)=\hat{p} (Y)$. This shows that $\hat{p}$ is an order homomorphism as well. 

Now we will show that $\hat{p}$ is surjective. Let $Z \in \O(P)$. Then $Z= \bigcup_{k=1} ^n  \dol z_k$. For each $z_k$ we find $x_k$ such that $p(x_k)=z_k$, in other words $\widehat{p}(\dol x_k)=\dol z_k$.  Then $$\hat{p}\left(\bigcup_{k=1} ^{n}  \dol x_k \right)=\bigcup_{k=1} ^{n} \hat{p}(\dol x_k)=\bigcup_{k=1} ^{n} \dol z_k = Z.$$ 

Now we will show that $\hat{p}$ is a strict poset homomorphism. Let $Y,Z \in \O(P)$ and $Y \subset Z$. 
Then $Y= \bigcup_{k=1} ^n  \dol y_k \subset Z= \bigcup_{k=1} ^m  \dol z_k$.  
Then for each $y_k$ we find $z_{l_k}$ such that $\dol y_k \subset \dol z_{l_k}$ which means that $y_k\leq z_{l_k}$. As $p$ is strict, there are $y_k',z_{l_k}'\in Q$ with $y_k'\leq z_{l_k}'$, $p(y_k')=y_k$ and $p(z_{l_k}')=z_{l_k}$.  

Put $Y':=\bigcup_{k=1} ^n  \dol y_k ^{'}$.  Then $\hat{p}\left( Y'\right)=Y$. Further for each $z_p$ such that $p \neq l_k$ for all $k=1,\ldots, n$ we find $z_p ^{''} \in Q$ such that $\hat{p}(\dol z_p ^{''})=\dol z_p$. Then $Y^{'} \subset Z'$ where $Z':= \bigcup_{k=1} ^n \dol z_{l_k} ^{'}  \cup  \bigcup \dol z_p ^{''}$ and the latter union is over all $p$ not in $\{l_k:k\leq n\}$. Since $\hat{p}(Z')=Z$, $\hat{p}$ is a quotient map for posets.
\end{proof}

It turns out that the value of mapping $\hat{p}(A)$ does not depend on how one represents $A$ as a union of principal down-sets.  
\begin{rem}\label{ValueOfIQMDoesntDependOnDecomposition}
Let $P,Q$ be finite posets, $p:Q \to P$ be a quotient map and $\hat{p}$ be a quotient map induced by $p$. Let $A \subset Q$ be a down-set with canonical decomposition $\{\dol x_i: i<n\}$ and let $\bigcup_{i=1} ^m \dol y_i=A$. Then $$\bigcup_{i=1} ^m \dol p(y_i)=\bigcup_{i=1} ^m \hat{p}(\dol y_i)=\hat{p}\left(\bigcup_{i=1} ^m \dol y_i\right)=\hat{p}(A)=\bigcup_{i=1} ^n \dol p(x_i).$$ It follows that if $\bigcup_{i=1} ^m \dol y_i=\bigcup_{i=1} ^l \dol z_i \subset Q$, then $\bigcup_{i=1} ^m \dol p(y_i)=\bigcup_{i=1} ^l \dol p(z_i)$.
\end{rem}

To properly define the inductive sequence of $\O(P_n)$'s such that diagram \eqref{DiagramToFill} commutes we need to show that $\hat{p}^n_k=\hat{p}^{k+1}_k\circ\dots\circ\hat{p}^{n}_{n-1}$ for all $n, k \in \N$ with $k<n$, which follows from the following.
\begin{lem}\label{InducedMappingsComposition}
		Let $P, Q, H$ be finite posets and $p:Q \to P$, $q:P \to H$ be quotient mappings. Then $\widehat{q \circ p}=\widehat{q} \circ \widehat{p}$.
\end{lem}
\begin{proof}
Let $A \in \O(Q)$. Consider the following cases:
\begin{itemize}
    \item if $A=\emptyset$, then $\widehat{q \circ p}(A)=\emptyset=\hat{q}(\emptyset)=\hat{q}(\hat{p}(\emptyset))$.
    
    \item if $A$ is a principal down-set $\dol x$, then $\widehat{q \circ p}(\dol x)=\dol q(p(x))=\hat{q}(\dol p(x))=\hat{q}(\hat{p}(\dol x))=(\hat{q} \circ \hat{p})(\dol x)$.
    
    \item if $A$ is non-empty, non-principal down-set, then let $\{\dol x_i:i<n\}$ be its canonical decomposition. Then $\widehat{q \circ p}(A)=\bigcup_{i=1} ^n \widehat{q \circ p}(\dol x_i)= \bigcup_{i=1}^n \hat{q}(\hat{p}(\dol x_i))=\hat{q}(\bigcup_{i=1} ^n \hat{p}(\dol x_i))=\hat{q}(\hat{p}(\bigcup_{i=1}^n \dol x_i))= (\hat{q} \circ \hat{p})(A)$.
\end{itemize}
Thus $\widehat{q \circ p}(A)=(\hat{q} \circ \hat{p})(A)$ for every $A \in \O(Q)$, which gives the assertion.
\end{proof}
Now, we provide the equivalent conditions for the induced mapping to be lattice homomorphism.
\begin{thm}
Let $P,Q$ be finite posets and $p: Q \to P$ be a quotient map. Then the following are equivalent.
\begin{enumerate}
    \item[(1)] $\hat{p}$ is a lattice-homomorphism.
    \item[(2)] $ \bigcup_{z \in \dol\: x \cap \dol\: y} \limits \dol p(z) \supset \dol p(x) \cap \dol p(y)$ for all $x,y \in Q$.
    \item[(3)] If $t \leq p(x)$ and $t \leq p(y)$, then there is $z \in Q$, such that $z \leq x, z \leq y$ and $t \leq p(z)$, for all $x,y \in Q$ and $t \in P$.
\end{enumerate}
\end{thm}
\begin{proof}
The equivalence of (2) and (3) is pretty clear, so we will only show that (1) and (2) are equivalent. Assume that (2) is true.
Take any $x, y \in Q$. Observe that 
$$\bigcup_{z \in \dol\: x \cap \dol\: y} \dol p(z) = \bigcup_{\dol\: z \subset \dol\: x \cap \dol\: y} \hat{p}(\dol z) \subset \hat{p}(\dol x \cap \dol y)$$
It follows that $\hat{p}(\dol x) \cap \hat{p}(\dol y)=\dol p(x) \cap \dol p(y) \subset \hat{p}(\dol x \cap \dol y)$. 
Since $\hat{p}$ is an order-homomorphism, then $\hat{p}(\dol x \cap \dol y) \subset \hat{p}(\dol x) \cap \hat{p}(\dol y)$.
Thus $\hat{p}(\dol x) \cap \hat{p}(\dol y)=\hat{p}(\dol x \cap \dol y)$
Now take any two non-empty, down-sets $A,B$. Then $A=\bigcup_{i=1} ^n \dol a_i, B=\bigcup_{i=1} ^m \dol b_i$. See that
$$\hat{p}(A \cap B)=\hat{p}(\bigcup_{i=1} ^n \dol a_i \cap \bigcup_{i=1} ^m \dol b_i)=\hat{p}(\bigcup_{i=1}^n \bigcup_{j=1}^m \dol a_i \cap \dol b_j)=\bigcup_{i=1}^n \bigcup_{j=1}^m \hat{p}(\dol a_i \cap \dol b_j)=$$
$$=\bigcup_{i=1}^n \bigcup_{j=1}^m \hat{p}(\dol a_i)  \cap \hat{p}(\dol b_j)=\bigcup_{i=1} ^n \hat{p}(\dol a_i) \cap \bigcup_{i=1} ^m \hat{p}(\dol b_i)=\hat{p}(A) \cap \hat{p}(B).$$
Note that $\hat{p}(A \cap \emptyset)=\emptyset=\hat{p}(A) \cap \hat{p}(\emptyset)$ for any down-set $A$, thus $\hat{p}$ is a lattice-homomorphism.

Now, assume (1). Take any $x, y \in Q$. 
We need to consider two cases
\begin{itemize}
    \item $\dol x \cap \dol y = \emptyset$
    \item $\dol x \cap \dol y \neq \emptyset$
\end{itemize}
If $\dol x \cap \dol y = \emptyset$, then $\hat{p}(\dol x \cap \dol y)=\hat{p}(\dol x) \cap \hat{p}(\dol y) = \emptyset$. Thus $\dol p(x) \cap \dol p(y)\subset \bigcup_{z \in \dol\: x \cap \dol\: y} \dol p(z)$. 

If $\dol x \cap \dol y \neq \emptyset$ and $\dol x \cap \dol y =\bigcup_{i=1}^m \dol z_i=$ where $z_i$'s form its canonical decomposition, then $\hat{p}(\dol x \cap \dol y) \neq \emptyset$. Let $\hat{p}(\dol x \cap \dol y)=\bigcup_{i=1} ^n \dol r_i$ be its canonical decomposition. Since $\bigcup_{i=1}^m \dol p(z_i)=\bigcup_{i=1}^n \dol r_i$, then $\{r_1, \dots, r_n\} \subset \{p(z_1), \dots, p(z_m)\}$. Without loss of generality we may assume that $r_i=p(z_i)$ for $i=1, \dots, n$. Finally $$\dol p(x) \cap \dol p(y)=\hat{p}(\dol x \cap \dol y)=\bigcup_{i=1}^n \dol p(z_i)=\bigcup_{z \in \dol\: x \cap \dol\: y} \limits \dol p(z),$$
which means that (2) holds.
\end{proof}

\subsection{Properties of $\mathbb{O}(P)$
}
Let $\langle\{P_n\}_{n \in \N},\{p^n_k\}_{k<n}\rangle$ be an inductive sequence of finite posets with quotient mappings. Let $\hat{p}_k ^n:\O(P_n) \to \O(P_k)$ be a quotient mapping induced by $p_k ^n$ for $n \in \N$ and $k < n$, and let $P=\varprojlim \langle\{P_n\}_{n \in \N},\{p^n_k\}_{k<n}\rangle$.  Lemma \ref{InducedMappingsComposition} shows that $\langle \{\O(P_n)\}_{n \in \N}, \{\hat{p}^n_k\}_{k<n} \rangle$ is an inductive sequence in the category of finite posets.
We will now study the propeties of $\mathbb{O}(P):=\varprojlim \langle \{\O(P_n)\}_{n \in \N}, \{\hat{p}^n_k\}_{k<n} \rangle$.
\begin{thm}\label{EmbeddingTheorem}
 Let $\psi(x_1, x_2, \ldots)=(\dol x_1, \dol x_2, \ldots)$ for $(x_1, x_2, \dots) \in P$. Then $\psi:P \to \mathbb{O}(P)$ is an order-embedding.
\end{thm}
\begin{proof}
 We will show that $\psi$ is an order-embedding into $\mathbb{O}(P)$.
Firstly we show that $\psi$ is well defined, more precisely it actually maps $P$ into $\mathbb{O}(P)$. Take $(x_1, x_2, \ldots) \in P$. We will show that $\hat{p}_n ^{n+1}(\dol x_{n+1})=\dol x_n$ for every $n \in \N$. Due to Theorem \ref{Posets_and_O(P)} diagram
\begin{equation}
\begin{tikzcd}
P_n \arrow[d, "\downarrow x"'] & P_{n+1} \arrow[l, "p_n ^{n+1}"'] \arrow[d, "\downarrow x"] \\
\O(P_n)					& \O(P_{n+1}) \arrow[l, "\hat{p}_n ^{n+1}"] \\
\end{tikzcd}
\end{equation}
commutes for every $n \in \N$, so $\hat{p}_n ^{n+1}(\dol x_{n+1})= \dol p_n ^{n+1}(x_{n+1})= \dol x_n$ for every $n \in \N$. Thus $\psi$ is properly defined.

Now we will show that it is an order embedding.
Let $(x_1, x_2, \ldots), (y_1, y_2, \ldots) \in P$. Then
$$(x_1, x_2, \ldots) \leq (y_1, y_2, \ldots) \iff x_n \leq y_n \text{ for every } n \in \N \iff$$ $$\dol x_n \subset \dol y_n \text{ for every } n \in \N \iff  
\psi(x_1, x_2, \dots)=(\dol x_1, \dol x_2, \dots) \leq (\dol y_1, \dol y_2, \dots)= \psi(y_1, y_2, \dots),$$ which makes $\psi$ an order embedding.
\end{proof}
\begin{Def}
We say that $\mathbf{x} \in \mathbb{O}(P)$ is a \emph{principal element} if $\mathbf{x}=(\dol a_1, \dol a_2, \dots)$ for some $(a_1, a_2, \dots) \in P$.
\end{Def}
We already know that $\mathbb{O}(P)$ is a poset (as an inverse limit of finite posets with quotients).  Surprisingly, despite the fact that induced quotient mapping does not preserve meets, see Example \ref{ExampleNotPreservingMeets}, it turns out that $\mathbb{O}(P)$ is actually a lattice.
\begin{thm}
$\mathbb{O}(P)$ is a lattice.
\end{thm}
\begin{proof}
We will show that for every $x, y \in \mathbb{O}(P)$ their supremum exists. 
Let $(x_1, x_2, \ldots)$, $(y_1, y_2, \ldots) \in \mathbb{O}(P)$. It is easy to see that $(z_1, z_2, \ldots):=(x_1 \cup y_1, x_2 \cup y_2, \ldots)$ is an upper bound of $\lbrace x, y \rbrace $ because it is an element of $\mathbb{O}(P)$ as $\hat{p}_n ^{n+1}$ preserves set unions and $x_n \subset z_n, y_n \subset z_n$ for every $n \in \N$. Now we will show it is actually the smallest upper bound. Let $(s_1 ,s_2, \ldots)$ be any upper bound of $\lbrace x, y \rbrace $. Then for every $n \in \N$ we have $x_n \subset s_n, y_n \subset s_n$ thus $x_n \cup y_n \subset s_n$ so $ (z_1, z_2, \ldots) \leq (s_1, s_2, \ldots)$. Therefore $(z_1, z_2, \ldots)$ is the supremum of $\lbrace x, y \rbrace $.

Now we will show that $\lbrace x, y \rbrace $ has an infimum. 
It is easy to see that $(\emptyset, \emptyset, \ldots)$ is a lower bound of $\lbrace x, y \rbrace $. Consider the set $$A_1:= \lbrace z_1 \in \O(P_1): \text{ there is a lower bound of  } \lbrace x, y \rbrace \text{ such that } z_1 \text{ is its first coordinate} \rbrace.$$ Clearly it is a non-empty set. Because supremum of two lower bonds of $\lbrace x, y \rbrace$ is still a lower bound of $\lbrace x, y \rbrace$ we get that $A_1$ has the greatest element, say $r_1$. Now consider the set $$A_2:= \lbrace z_2 \in \O(P_2): \text{ there is a lower bound of  } \lbrace x, y \rbrace \text{ such that } r_1 \text{ is its first coordinate and } $$ $$ z_2 \text{ is the second} \rbrace.$$ Similarly this set is non-empty and it has the greatest element $r_2$. There is a lower bound of $\{x,y\}$ of the form $(r_1,r_2,z_3,z_4,\dots)\in\mathbb{O}(P)$, which in particular means that $\hat{p}^{2}_1(r_{2})=r_1$. 
Proceeding inductively we define $A_n$'s and a sequence $(r_1, r_2, \ldots)$. By the construction $\hat{p}^{n+1}_n(r_{n+1})=r_n$ for all $n \in \N$, which means $(r_1, r_2, \ldots) \in \mathbb{O}(P)$. Furthermore it is a lower bound of $\lbrace x, y \rbrace $.

Assume that $(s_1, s_2, \ldots)$ is a lower bound of $\{x,y\}$. Then by the first part of the proof $(s_1\cup r_1,s_2\cup r_2,\ldots)$ is a lower bound of $\{x,y\}$ as well and $r_i\leq s_i\cup r_i$. Thus by the definition of $A_i$'s and simple induction we get $r_i=s_i\cup r_i$. Thus $s_i\subset r_i$, and consequently $(s_1,s_2,\ldots)\leq(r_1,r_2,\ldots)$.
\end{proof}
Let us present some properties of a limit lattice $\mathbb{O}(P)$. Firstly note that $(\emptyset,\emptyset,\ldots)$ is zero (or the least) element of $\mathbb{O}(P)$. Moreover, $(P_1,P_2,\ldots)$ is one (or the greatest) element of $\mathbb{O}(P)$. As we have already proved $\sup\{(x_1,x_2,\ldots),(y_1,y_2,\ldots)\}=(x_1\cup y_1,x_2\cup y_2,\ldots)$. On the other hand there is not such easy formula for infimum. We will focus on more involved properties of $\mathbb{O}(P)$.  
\begin{thm}
	$\mathbb{O}(P)$ is atomic.
\end{thm}
\begin{proof}
	Take any $(x_1, x_2, \dots) \in \mathbb{O}(P) \setminus \{(\emptyset, \emptyset, \dots)\}$. We will find an atom $(a_1, a_2, \dots)$ , such that $(a_1, a_2, \dots) \leq (x_1, x_2, \dots)$. 
	If $(x_1, x_2, \dots)$ is an atom, then we can take $(a_1, a_2, \dots):=(x_1, x_2, \dots)$ and we get the required condition.
	Assume that $(x_1, x_2, \dots)$ is not an atom. Then there is $(y_1, y_2, \dots)\neq (\emptyset, \emptyset, \dots)$ with $(y_1, y_2, \dots)<(x_1, x_2, \dots)$.
	
	See that then set $$A_1:=\{z_1 \in \O(P_1)\setminus \{\emptyset\}: \text{there is a non-zero element of } \mathbb{O}(P) \text{ that is strictly less than }$$ 
	$$(x_1, x_2, \dots) \text{ and } z_1 \text{ is its first coordinate.}\}$$ is non-empty. As a non-empty subset of finite poset, it has a minimal element, say $a_1$. Thus we can find $(a_1, z_2, z_3, \dots) \in \mathbb{O}(P)$ with $(\emptyset, \emptyset, \dots)<(a_1, z_2, z_3, \dots)<(x_1, x_2, \dots)$. Thus we get that set
	$$A_2:=\{z_2 \in \O(P_2)\setminus \{\emptyset\}: \text{there is a non-zero element of } \mathbb{O}(P) \text{ that is strictly less than }$$  
	$$(x_1, x_2, \dots) \text{ and } a_1, z_2 \text{ are its first and second coordinate, respectively.}\}$$
	is non-empty. Similarly, as a non-empty subset of a finite poset, it has a minimal element, say $a_2$. Proceeding inductively, we define $A_n$'s and sequence $(a_1, a_2, \dots)$. Clearly, $\hat{p}_n ^{n+1}(a_{n+1})=a_n$ for each $n \in \N$. By the construction, we get that $(a_1, a_2, \dots) \leq (x_1, x_2, \dots)$.
	
	Now we will prove, that $(a_1, a_2, \dots)$ is an atom. Suppose that there is $(s_1, s_2, \dots) \in \mathbb{O}(P)$ such that $(\emptyset, \emptyset, \dots) \neq (s_1, s_2, \dots)<(a_1, a_2, \dots)$.
	Let $i$ be the smallest index such that $s_j=a_j$ for $j<i$, and $s_i<a_i$. Thus $s_i\in A_i$. But this contradicts the definition of $a_i$. 
\end{proof}

The following Lemmas show the structural properties of mappings induced by quotients. 

\begin{lem}\label{ExtendingLemma}
Let $Q,H$ be finite posets and $g:\O(Q) \rightarrow \O(H)$ be a quotient map induced by quotient map from $Q$ onto $H$. Let $q\in\O(Q)$. Then for every element $\dol a$ of the canonical decomposition of $g(q)$ we can find an element $\dol b$ of the canonical decomposition of $q$ such that $g(\dol b)=\dol a$.
\end{lem}
\begin{proof}
Let $p=g(q)$ and $\{\dol p_i: i<m \},\
\{\dol q_i: i<n \}$ be the canonical decompositions of $p$ and $q$, respectively. Thus $p=\bigcup_{i<m}  \dol p_i$ and $q=\bigcup_{i<n}  \dol q_i$. We know that $g$ preserves  set unions, so $p=g(q)=g\left(\bigcup_{i<n} \dol q_i \right)=\bigcup_{i<n}  g(\dol q_i)$. Thus $\bigcup_{i<m} \dol p_i=\bigcup_{i<n} g(\dol q_i)$.
Then $g(\dol q_{i})$ is a principal down-set for every $i<n$. Using Lemma \ref{CanonicalDecompositionLemma} we get that $\{\dol p_i: i<m\} \subset \{\dol g(q_i): i<n \}$. Thus for every $\ell<m$ we find $k<n$ such that $\dol p_\ell=g(\dol q_{k})$
\end{proof}
\begin{lem}\label{HelpingLemma1}
Let $\mathbf{p}=(p_1, p_2, \ldots) \in \mathbb{O(\text{P})}$. Let $m\in\N$ and let $t$ be a principal down-set in the canonical decomposition of $p_m$. Then there is a principal element $(\dol x_1, \dol x_2, \ldots)$ such that 
\begin{enumerate}
    \item[(i)] $\dol x_i$ is in canonical decomposition of $p_i$ for every $i > m$;
    \item[(ii)] $\dol x_m=t$,
    \item[(iii)] $\dol x_i \leq p_i$ for every $i<m$.
\end{enumerate}
In particular $(\dol x_1, \dol x_2, \ldots) \leq \mathbf{p}$.
\end{lem}
\begin{proof}
Fix a principal down-set $t$ in the canonical decomposition of $p_m$. Then there is $x_m$ in $P_m$ with $t=\dol x_m$, which gives \textit{(ii)}.  Condition \emph{(i)}  immediately follows from Lemma \ref{ExtendingLemma} by a simple induction. By the definition of induced quotient map we obtain \emph{(iii)}.
\end{proof}

Now we present a lemma that allows us to represent elements of $\mathbb{O}(P)$ as supremum of countably many principal elements $(\dol x_1, \dol x_2, \ldots)$.
\begin{lem}\label{LatticeElementDecompositionLemma}
For any non-zero element $\mathbf{a} \in \mathbb{O}(P)$ there is a countable family $\mathcal{A}$ consisting of principal elements, such that $\sup(\mathcal{A})=\mathbf{a}$.
\end{lem}
\begin{proof}
Let $\mathbf{a}=(a_1, a_2, \ldots) \in \mathbb{O}(P)$ be non-zero. Each coordinate of $\mathbf{a}$ can be expressed as a finite union of principal down-sets from its canonical decomposition. Thus $a_n = \bigcup_{i=0} ^{m_n} \dol x_n ^{(i)}$ for every $n \in \N$ where $\{\dol x_n ^{(i)}:i<m_n \}$ is the canonical decomposition of $a_n$. 
For every $\dol x_1^{(i)}$, using Lemma \ref{HelpingLemma1}, we can fix a principal element $\mathbf{b}_1^{(i)}$ in $\mathbb{O}(P)$, such that $\dol x_1^{(i)}$ is its first coordinate and $\mathbf{b}_1^{(i)} \leq \mathbf{a}$. Let $A_1=\{\mathbf{b}_1^{(i)}:i\leq m_1\}$. We already know that supremum of two elements from $\mathbb{O}(P)$ is the sequence of suprema, and therefore the first coordinate of $\sup(A_1)$ is $a_1$.  
For every $\dol x_2^{(i)}$ we can find a principal element $\mathbf{b}_2 ^{(i)}$ in $\mathbb{O}(P)$ such that $x_2^{(i)}$ is its second coordinate and $\mathbf{b}_2 ^{(i)} \leq \mathbf{a}$. Let $A_2=\{\mathbf{b}_2^{(i)}:i\leq m_2\}$. Then the second coordinate of $\sup(A_2)$ is $a_2$ Proceeding inductively we obtain $A_n$'s. Now we define $\mathcal{A}$ as $\bigcup_{n \in \N}  A_n$. Note that as every $A_n$ is finite, $\mathcal{A}$ is countable. 

Now we show that $\sup(\mathcal{A})=\mathbf{a}$. By the definition of $\mathcal{A}$, $\mathbf{x} \leq \mathbf{a}$ for all $\mathbf{x} \in \mathcal{A}$ so $\mathbf{a}$ is an upper bound of $\mathcal{A}$. Take any other upper bound $s=(s_1, s_2, \ldots)$ of $\mathcal{A}$. Fix $n\in\N$. Then $y_n\subset s_n$ for every $\mathbf{y}\in\mathcal{A}$. By the construction $\bigcup\{y_n:\mathbf{y}\in\mathcal{A}\}$ equals $a_n$.
Thus $\mathbf{a} \leq s$ and $\sup(\mathcal{A})=\mathbf{a}$. \end{proof}

Now we will show that $\mathbb{O}(P)$ is closed on taking supremum of countably many principal elements.
\begin{lem}\label{CountableSupremumLemma}
Let $\mathcal{A}:=\{(\dol x_1 ^{(i)}, \dol x_2 ^{(i)}, \ldots) \in \mathbb{O}(P):i \in \w \}$. Then $$\sup(\mathcal{A})=\left(\bigcup_{i \in \w} \dol x_1 ^{(i)}, \bigcup_{i \in \w} \dol x_2 ^{(i)}, \ldots\right).$$
\end{lem}

\begin{proof}
We begin by proving that $(\bigcup_{i \in \w} \dol x_1 ^{(i)}, \bigcup_{i \in \w} \dol x_2 ^{(i)}, \ldots) \in \mathbb{O}(P)$. Since $\O(P_n)$ is finite, there is $m_n$ with $\bigcup_{i<m_n}  \dol x_n ^{(i)}=\bigcup_{i \in \w} \dol x_n ^{(i)}$. We may assume $m_{n} \leq m_{n+1}$ for all $n \in \N$. 
Then 
\[\hat{p}_n ^{n+1}\left(\bigcup_{i \in \w} \dol x_{n+1} ^{(i)}
\right)=\hat{p}_n ^{n+1}\left(\bigcup_{i<m_{n+1}}  \dol x_{n+1} 
^{(i)}\right)=\bigcup_{i<m_{n+1}}  \hat{p}_n ^{n+1} \left( \dol 
x_{n+1} ^{(i)}\right)=\bigcup_{i<m_{n+1}} \dol x_{n} 
^{(i)}=\bigcup_{i \in \w} \dol x_n ^{(i)}.\]  So $
\left(\bigcup_{i \in \w} \dol x_1 ^{(i)}, \bigcup_{i \in 
\w}  \dol x_2 ^{(i)}, \ldots\right) \in \mathbb{O}(P)$.

Clearly $\left(\bigcup_{i \in \w}  \dol x_1 ^{(i)}, \bigcup_{i \in \w} \dol x_2 ^{(i)}, \ldots\right) $ is an upper bound of $\mathcal{A}$ so we only need to show that it is the least upper bound. Take any upper bound $s=(s_1, s_2, \ldots) \in \mathbb{O}(P)$ of $\mathcal{A}$. Then $\dol x_n ^{(i)} \subset s_n$ for all $n \in \N, i \in \w$. Thus $\bigcup_{i \in \w}  \dol x_n ^{(i)} \subset s_n$ for all $n \in \N$. So $\left(\bigcup_{i \in \w} \dol x_1 ^{(i)}, \bigcup_{i \in \w}  \dol x_2 ^{(i)}, \ldots\right) \leq (s_1, s_2, \ldots)$, and finally  $\sup(\mathcal{A})=\left(\bigcup_{i \in \w}  \dol x_1 ^{(i)}, \bigcup_{i \in \w}  \dol x_2 ^{(i)}, \ldots\right)$.
\end{proof}
The following is an immediate consequence of Lemma \ref{LatticeElementDecompositionLemma} and Lemma \ref{CountableSupremumLemma}. 
\begin{cor}
For any non-zero element $\mathbf{a} \in \mathbb{O}(P)$ there are $(x_1^{(i)},x_2^{(i)},\dots)$ in $P$, $i\in\w$, such that $\mathbf{a}=(\bigcup_{i \in \w} \dol x_1 ^{(i)}, \bigcup_{i \in \w} \dol x_2 ^{(i)}, \ldots)$.
\end{cor}
We say that lattice $L$ is $\sigma$-complete, if for any countable set $A \subset L$, both $\sup(A)$ and $\inf(A)$ exist. Now, we are ready to formulate and prove the next property of $\mathbb{O}(P)$. 
\begin{thm}
$\mathbb{O}(P)$ is $\sigma$-complete.
\end{thm}
\begin{proof}
	Let $\mathcal{A}:=\{(a_1 ^{(i)}, a_2 ^{(i)}, \dots), i \in \w\}$. From Lemma \ref{LatticeElementDecompositionLemma} and Lemma \ref{CountableSupremumLemma} we get that $\bigvee \mathcal{A} =\bigvee_{i \in \w} \sup(\mathcal{A}_i)$, where $\sup(\mathcal{A}_i)=(a_1 ^{(i)}, a_2 ^{(i)}, \dots)$ and $\mathcal{A}_i$ is a countable family consisting of principal elements. Now from Lemma \ref{LemmaOnSupremumOfUnions} we get that $\bigvee_{i \in \w} \sup(\mathcal{A}_i)= \sup(\bigcup_{i \in \w} \mathcal{A}_i)$, because $\sup(\bigcup_{i \in \w} \mathcal{A}_i)$ and $\sup(\mathcal{A}_i)$ exist as they are countable sets consisting of principal elements. 
	
	Now we are going to show that $\bigwedge \mathcal{A} \in \mathbb{O}(P)$.
	Consider the sequence $x_n=(a_1 ^{(0)}, a_2 ^{(0)}, \dots) \wedge (a_1 ^{(1)}, a_2 ^{(1)}, \dots) \wedge \dots \wedge (a_1 ^{(n)}, a_2 ^{(n)}, \dots)$. We can find $\ell_0$ such that for every $n \geq \ell_0$ the first coordinate of $x_n$ is constant. We denote this constant coordinate by $r_1$. Similarly we can find $\ell_1 \geq \ell_0$ such that for every $n \geq \ell_1$ the first coordinate of $x_n$ is $r_1$ and the second one is constant. We denote this constant second coordinate by $r_2$. Proceeding inductively we get a sequence $(r_1, r_2, \dots)$. 
	
	Now we will show that  $\bigwedge \mathcal{A}=(r_1, r_2, \dots)$. See that for every $n \in \N$, $\hat{p}_n ^{n+1}(r_{n+1})=r_n$ as $x_{\ell_{n}} \in \mathbb{O}(P)$. Thus $(r_1, r_2, \dots) \in \mathbb{O}(P)$. Note that for every $n \in \w$ we have $(r_1,r_2, \dots) \leq x_n$, which implies $(r_1,r_2, \dots) \leq (a_1 ^{(n)}, a_2 ^{(n)}, \dots)$. In other words, $(r_1, r_2, \dots)$ is a lower bound of $\mathcal{A}$. Let $(s_1, s_2, \dots)$ be a lower bound of $\mathcal{A}$. Then $(s_1, s_2, \dots) \leq x_n$ for every $n \in \w$. By the definition of $r_i$'s we obtain $s_i \subset r_i$, thus $(s_1, s_2, \ldots)\leq(r_1, r_2, \ldots)$.  
	Thus $\bigwedge \mathcal{A}=(r_1, r_2, \dots)$.
\end{proof}
Finally, we are able to prove that $\mathbb{O(P)}$ is universal with respect to induced quotient mappings, i.e. every inverse limit of order ideals $\mathbb{O}(H)=\varprojlim \langle \{\O(H_n)\}_{n \in \N}, \{\hat{h}^n_k\}_{k<n} \rangle$ is a continuous quotient of $\mathbb{O(P)}$.
\begin{thm}\label{InducedQuotientMapping}
Let $H= \varprojlim \langle  \{H_n\}_{n \in \N},\{h^n_k \}_{k < n} \rangle$ and $\mathbb{P}= \varprojlim \langle  \{P_n\}_{n \in \N},\{p^n_k \}_{k < n} \rangle$, where $(P_n)$ is a \text{\Fraisse} sequence and let $f:\mathbb{P}\to H$ be a continuous quotient map.
Then there is a continuous quotient map $q:\mathbb{O(P)} \rightarrow \mathbb{O(\text{H})}$ which is join-preserving. Furthermore the diagram
\begin{equation} \label{LastDiagram}
\begin{tikzcd}
H \arrow[d, "\psi"]  & \mathbb{P} \arrow[l, "f"']  \arrow[d, "\phi"] \\
\mathbb{O}(H)					& \mathbb{O(P)}  \arrow[l,"q"]  \\
\end{tikzcd}
\end{equation}
commutes.
\end{thm}
\begin{proof}

The quotient map $f$ has the following representation (see Theorem \ref{RepresentationOfContinuousQuotientMapping}): there are a sequence $i_1<i_2<\dots$ and quotient mappings $f_k: P_{i_k}\to H_k$ such $f(x_1,x_2,\dots)=(f_1(x_{i_1}), f_2(x_{i_2}),\dots)$ and diagram \eqref{RepresentationDiagram} commutes for every $n, k \in\N$ with $k<n$.

Firstly we define $q$ for principal elements as follows: $q(\dol x_1, \dol x_2, \ldots):=\psi(f( x_1,  x_2, \ldots))=(\dol f_1(x_{i_1}), \dol f_2(x_{i_2}), \ldots)$.
Also we put $q(\emptyset, \emptyset, \ldots)=(\emptyset, \emptyset, \ldots)$. Finally for arbitrary $\mathbf{a}=(a_1, a_2, \ldots)$ using Lemma \ref{LatticeElementDecompositionLemma} we find a countable family $\mathcal{A}=\{(\dol x_1 ^{(t)},\dol x_2 ^{(t)}, \ldots ):t\in\w \}$ such that  $\mathbf{a}=\sup\mathcal{A}$ and we put \[q(\mathbf{a}):=\sup\{q(\dol x_1 ^{(t)},\dol x_2 ^{(t)}, \ldots):t\in\w \}=\sup\{(\dol f_1(x_{i_1}^{(t)}), \dol f_2(x_{i_2} ^{(t)}), \ldots):t\in\w \}.\]

Using Lemma \ref{CountableSupremumLemma} we get that in fact
$$q(\mathbf{a})=\left(\bigcup_{t \in \w} \dol f_1(x_{i_1} ^{(t)}), \bigcup_{t \in \w} \dol f_2(x_{i_2} ^{(t)}), \dots \right)$$
Clearly diagram \eqref{LastDiagram} commutes.
Now we will prove that the value of $q$ does not depend on the choice of family $\mathcal{A}$ for $\mathbf{a}$. Assume that  \[\mathbf{a}=(a_1,a_2,\dots)=\sup\{(\dol x_1 ^{(t)},\dol x_2 ^{(t)}, \ldots):t\in\w \}=\sup\{(\dol y_1 ^{(t)},\dol y_2 ^{(t)}, \ldots):t\in\w \}.\] We need to show that $$\sup \{(\dol f_1(x_{i_1} ^{(t)}),\dol f_2(x_{i_2} ^{(t)}), \ldots):t\in\w \}=\sup \{(\dol f_1(y_{i_1} ^{(t)}),\dol f_2(y_{i_2} ^{(t)}), \ldots):t\in\w \}$$ which is, by  Lemma \ref{CountableSupremumLemma} equivalent to 
$$
\bigcup_{t \in \w} \dol f_l(x_{i_l} ^{(t)})=\bigcup_{t \in \w} \dol f_l(y_{i_l} ^{(t)}) \text{ for all } l \in \N.
$$

Take any $l \in \N$. Again by Lemma we obtain $\bigcup_{t \in \w} \dol x_{i_l} ^{(t)}=\bigcup_{t \in \w}  \dol y_{i_l} ^{(t)}$. Since $H_{i_l}$ is finite, we find $m_l,n_l \in \mathbb{N}$ with 
$$\bigcup_{t \in \w} \dol x_{i_l} ^{(t)}=\bigcup_{t<m_l} \dol x_{i_l} ^{(t)}\text{ and }\bigcup_{t<n_l}  \dol y_{i_l} ^{(t)}=\bigcup_{t \in \w}  \dol y_{i_l} ^{(t)} $$
From Remark \ref{ValueOfIQMDoesntDependOnDecomposition} we get that $\bigcup_{t<m_l} \dol f_l(x_{i_l}^{(t)})=\bigcup_{t<n_l} \dol f_l(y_{i_l}^{(t)})$. Note that 
$$
\bigcup_{t \in \w} \dol f_l(x_{i_l} ^{(t)})=\bigcup_{t <m_l} \dol f_l(x_{i_l} ^{(t)})=\bigcup_{t <n_l} \dol f_l(y_{i_l} ^{(t)})=\bigcup_{t \in \w} \dol f_l(y_{i_l} ^{(t)}).
$$

Thus $q$ is well defined.

Now we will show that $q$ preserves suprema. Let $\mathbf{a}=(a_1, a_2, \ldots),\mathbf{b}=(b_1, b_2, \ldots)$.
We know that both of these elements are countable supremum of principal elements $( \dol x_1 ^{(t)}, \dol x_2 ^{(t)}, \dots)$ and $( \dol y_1 ^{(t)}, \dol y_2 ^{(t)}, \dots)$, respectively ($t\in\w$). As value of $q$ does not depend on a chosen decomposition of $\sup(\{\mathbf{a},\mathbf{b}\})$, we have that
$$q(\sup(\{\mathbf{a},\mathbf{b}\}))=q\left(\bigcup_{t \in \w} \dol x_{1} ^{(t)} \cup \bigcup_{t \in \w} \dol y_{1} ^{(t)}, \bigcup_{t \in \w} \dol x_{2} ^{(t)} \cup \bigcup_{t \in \w} \dol y_{2} ^{(t)}, \ldots \right)=
$$
$$\left(\bigcup_{t \in \w} \dol f_1(x_{i_1} ^{(t)}) \cup \bigcup_{t \in \w} \dol f_1(y_{i_1} ^{(t)}), \bigcup_{t \in \w} \dol f_2(x_{i_2} ^{(t)}) \cup \bigcup_{t \in \w} \dol f_2(y_{i_2} ^{(t)}), \ldots \right)=\sup(\{q(\mathbf{a}), q(\mathbf{b})\})$$
Thus $q$ preserves suprema. And consequently, it  preserves ordering as well.  Also $q$ is onto $\mathbb{O}(H)$. Indeed, take any $\mathbf{g} \in \mathbb{O}(H)$.
Then let $\mathcal{G}:=\{(\dol g_1 ^{(t)}, \dol g_2 ^{(t)}, \dots):t \in \w \}$ be a family of principal elements in $\mathbb{O}(H)$ such that $\sup(\mathcal{G})=\mathbf{g}$. Observe that each element $\mathbf{x}$ of family $\mathcal{G}$ can be expressed as a value $\mathbf{x}=q(\mathbf{a})=\psi(f(\phi^{-1}(\mathbf{a})))$ of some principal element $\mathbf{a}$ from $\mathbb{O(P)}$ as:
\begin{enumerate}
\item[(i)]$\psi$ is an isomorphism from $H$ to the set of all principal elements in $\mathbb{O}(H)$.
\item[(ii)]$f$ is a quotient map from $\mathbb{P}$ to $H$.
\item[(iii)]$\phi ^{-1}$ is an isomorphism from the set of all principal elements in $\mathbb{O(P)}$ to $\mathbb{P}$.
\end{enumerate}
Isomorphism is onto mapping and the composition of a onto mappings is an onto mapping. Thus $q$ is a mapping from the set of all principal elements in $\mathbb{O(P)}$ onto the set of all principal elements in $\mathbb{O}(H)$.
For each element $(\dol g_1 ^{(t)}, \dol g_2 ^{(t)}, \dots)$ in $\mathcal{G}$ we can find an  element $(\dol w_1 ^{(t)}, \dol w_2 ^{(t)}, \dots)$ such that $q(\dol w_1 ^{(t)}, \dol w_2 ^{(t)}, \dots)=(\dol g_1 ^{(t)}, \dol g_2 ^{(t)}, \dots)$. Let $\mathcal{W}:=\{(\dol w_1 ^{(t)}, \dol w_2 ^{(t)}, \dots) \in \mathbb{O(P)}:t \in \w \}$. Clearly $q(\sup(\mathcal{W}))=\mathbf{g}$.

Now we must prove that $q$ is a quotient map for posets. Take $\mathbf{a}, \mathbf{b} \in \mathbb{O}(H)$ such that $\mathbf{a} \leq \mathbf{b}$. Let $\mathcal{A},\mathcal{B}$ be families of principal elements such that $\mathbf{a}=\sup(\mathcal{A})$ and $\mathbf{b}=\sup(\mathcal{B})$. 

We will show that for any element $(\dol a_1 ^{(t)}, \dol a_2 ^{(t)}, \ldots)$ of $\mathcal{A}$ we will find $( \dol b_1 ^{(m_1)}, \dol b_2 ^{(m_2)}, \ldots)$ such that $(\dol a_1 ^{(t)}, \dol a_1 ^{(t)}, \ldots) \leq (\dol b_1 ^{(m_1)}, \dol b_2 ^{(m_2)}, \ldots)$ and $(\dol b_n^{(m_i)})_{n\in\N}$ are elements of $\mathcal{B}$ for every $i\in\N$.

Take any $t \in \w$. Then $(\dol a_1 ^{(t)}, \dol a_2 ^{(t)}, \ldots) \leq ( \bigcup_{t \in \w}  \dol b_1 ^{(t)},  \bigcup_{t \in \w} \dol b_2 ^{(t)}, \ldots )$. Take any $n \in \N$. Then there is $m_n \in \w$ (depending on $n$) such that $\dol a_n ^{(t)} \subset \dol b_n ^{(m_n)}$. 

This proves that set (recall that $t$ is fixed!) $$G_{n}:=G_n^{(t)}=\{(\dol a_n ^{(t)},\dol b_n ^{(m)}) \in \O(H_n)^2:\dol a_n ^{(t)} \subset \dol b_n ^{(m)}\text{ for some }m\}$$ 
is non-empty for all $n \in \N$. It is also finite as a subset of a finite set.

By $T_{n}$ we denote that set of all elements $((\dol a_1 ^{(t)},\dol b_1 ^{(m_1)}),\ldots, (\dol a_n ^{(t)},\dol b_n ^{(m_n)})) \in G_{1}\times \ldots \times G_{n}$ such that 
\[(\hat{h}_u ^n (\dol a_n ^{(t)}), \hat{h}_u ^n (\dol b_n ^{(m_n)}))=
(\dol a_u ^{(t)},\dol b_u ^{(m_u)}) \text{ for any } u < n. \]
Since $\hat{h}_u ^n$ is a homomorphism, $\dol a_n ^{(t)} \subset \dol b_n ^{(m_n)}$  implies $\dol a_u ^{(t)} \subset \dol b_u ^{(m_u)}$ for every $u < n$. Therefore $T_{n}$ is non-empty.

Using Lemma \ref{FiniteSequencesLemma}, for $q^n_u:G_n\to G_u$ given by $q^n_u(\dol a_n ^{(t)},\dol b_n ^{(m_n)})=(\hat{h}_u ^n (\dol a_n ^{(t)}), \hat{h}_u ^n (\dol b_n ^{(m_n)}))$, we get that there is a sequence $((\dol a_n ^{(t)}, \dol b_n ^{(m_n)})) \in \prod_{n \in \N} G_{n}$ such that $\hat{h}_n ^{n+1} (\dol a_{n+1} ^{(t)})= a_n ^{(t)}$ and $\hat{h}_n ^{n+1}(\dol b_{n+1} ^{(m_{n+1})})=b_n ^{(m_n)}$ for all $n \in \N$.
Then $(\dol b_1^{(m_1)},\dol b_2^{(m_2)},\dots) \leq \mathbf{b}$.
So, by the definition of $G_n^{(t)}$'s, we obtain that for every $t \in \w$ there is $(m_n)\in \w^{\N}$ such that $(\dol a_1 ^{(t)}, \dol a_2 ^{(t)}, \ldots) \leq (\dol b_1 ^{(m_1)}, \dol b_2 ^{(m_2)}, \ldots)$. (We should write $(m_n^{(t)})$, since the sequence $(m_n)$ depends on $t$. But we want our proof to be readable).

As $q$ is a quotient map for principal elements (this follows immediately from the fact that $f$ is quotient), thus we will find $(\dol c_1 ^{(t)}, \dol c_2 ^{(t)}, \ldots)$ and $(\dol d_1 ^{(m_1)}, \dol d_2 ^{(m_2)}, \ldots)$ in $\mathbb{O(P)}$ such that $(\dol c_1 ^{(t)}, \dol c_2 ^{(t)}, \ldots) \leq (\dol d_1 ^{(m_1)}, \dol d_2 ^{(m_2)}, \ldots)$ and $q(\dol c_1 ^{(t)}, \dol c_2 ^{(t)}, \ldots)=(\dol a_1 ^{(t)}, \dol a_2 ^{(t)}, \ldots), q(\dol d_1 ^{(m_1)}, \dol d_2 ^{(m_2)}, \ldots)=(\dol b_1 ^{(m_1)}, \dol b_2 ^{(m_2)}, \ldots)$. Let $\mathcal{C}$ be the set of all $(\dol c_1 ^{(t)}, \dol c_2 ^{(t)}, \ldots)$ and $\mathcal{D}$ be the set of all corresponding $(\dol d_1 ^{(m_1)}, \dol d_2 ^{(m_2)}, \ldots)$.

We know that for each $(\dol b_1 ^{(m)}, b_2 ^{(m)}, \dots) \in \mathcal{B}$ we can find $(\dol e_1^{(m)}, \dol e_2^{(m)}, \dots) \in \mathbb{O(P)}$ such that $q(\dol e_1^{(m)}, \dol e_2^{(m)}, \dots)=(\dol b_1 ^{(m)}, b_2 ^{(m)}, \dots)$. Let $\mathcal{E}$ be the set of all corresponding $(\dol e_1^{(m)}, \dol e_2^{(m)}, \dots)$. Note that $\sup(\mathcal{C}) \leq \sup(\mathcal{D}) \leq \sup(\mathcal{D}) \vee \sup(\mathcal{E})$ and $q(\sup(\mathcal{C}))=\mathbf{a}$. Also $q(\sup(\mathcal{D}) \vee \sup(\mathcal{E}))=q(\sup(\mathcal{D})) \vee q(\sup(\mathcal{E}))=\sup(q[\mathcal{D}]) \vee \sup(q[\mathcal{E}])=\sup(\mathcal{B})=\mathbf{b}$.

At last we show that $q$ is continuous. Take any basic set $U_{(z_1, z_2, \dots, z_n)} \cap \mathbb{O}(H)$ and let $(r_1, r_2, \dots) \in q^{-1}[U_{(z_1, z_2, \dots, z_n)} \cap \mathbb{O}(H)]=q^{-1}[U_{(z_1, z_2, \dots, z_n)}] \cap \mathbb{O(P)}$. Consider the following cases:
\begin{itemize}
    \item[(i)] if $(r_1, r_2, \dots)$ is zero, then immediately from the definitions of $q$ and induced mapping we obtain   
    $$q^{-1}[U_{(z_1, z_2, \dots, z_n)} \cap \mathbb{O}(H)]=q^{-1}[\{(\emptyset, \emptyset, \dots) \}]=\{(\emptyset, \emptyset, \dots) \}.$$ Note that $$\{(\emptyset, \emptyset, \dots) \}=(\{\emptyset \} \times P_2 \times \dots) \cap \mathbb{O(P)}$$  so this set is open in $\mathbb{O(P)}$ and $(r_1, r_2, \dots) \in \{(\emptyset, \emptyset, \dots) \} \subset q^{-1}[U_{(z_1, z_2, \dots, z_n)}] \cap \mathbb{O(P)}$.
    
    \item[(ii)] if $(r_1, r_2, \dots)$ is a non-zero element, then take a family of principal elements $\{(\dol x_1 ^{(t)}, \dol x_2 ^{(t)}, \dots)\}_{t \in \w}$ such that $(r_1, r_2, \dots)=\sup\{(\dol x_1 ^{(t)}, \dol x_2 ^{(t)}, \dots): t \in \w\}$. We will prove that $(\{r_1\} \times \{r_2\} \times \dots \times \{r_{i_n}\} \times P_{i_n +1} \times \dots) \cap \mathbb{O(P)} \subset q^{-1}[U_{(z_1, z_2, \dots, z_n)} \cap \mathbb{O}(H)]$. Let $(l_1, l_2, \dots) \in (\{r_1\} \times \{r_2\} \times \dots \times \{r_{i_n}\} \times P_{i_n +1} \times \dots) \cap \mathbb{O(P)}$ and take a family of principal elements $\{(\dol y_1^{(t)}, \dol y_2^{(t)}, \dots): t \in \w \}$ with $\sup\{(\dol y_1^{(t)}, \dol y_2^{(t)}, \dots): t \in \w \}=(l_1, l_2, \dots)$. Then 
    $$\bigcup_{t \in \w} \dol y_{i_k} ^{(t)}=l_{i_k}=r_{i_k}=\bigcup_{t \in \w} \dol x_{i_k} ^{(t)} \text{ for } k \leq n,$$ so
    $$z_k=\bigcup_{t \in \w} \dol f_{k}(x_{i_k}^{(t)})=\bigcup_{t \in \w} \dol f_{k}(y_{i_k}^{(t)}) \text{ for } k \leq n.$$
    Since $$q(l_1, l_2, \dots)=(\bigcup_{t \in \w} \dol f_{1}(y_{i_1}^{(t)}), \bigcup_{t \in \w} \dol f_{2}(y_{i_2}^{(t)}), \dots, \bigcup_{t \in \w} \dol f_{n}(y_{i_n}^{(t)}), \bigcup_{t \in \w} \dol f_{n+1}(y_{i_{n+1}}^{(t)}), \dots)=$$ $$(z_1, z_2, \dots, z_n, \bigcup_{t \in \w} \dol f_{n+1}(y_{i_{n+1}}^{(t)}), \dots)$$ we get $(l_1, l_2, \dots) \in q^{-1}[U_{(z_1, z_2, \dots, z_n)} \cap \mathbb{O}(H)]$, which means $$(\{r_1\} \times \{r_2\} \times \dots \times \{r_{i_n}\} \times P_{i_n +1} \times \dots) \cap \mathbb{O(P)} \subset q^{-1}[U_{(z_1, z_2, \dots, z_n)} \cap \mathbb{O}(H)].$$
\end{itemize}
This shows that $q^{-1}[U_{(z_1, z_2, \dots, z_n)} \cap \mathbb{O}(H)]$ is a neighborhood of all its points, which makes it open. Finally $q$ is continuous.

\end{proof}
\section{The category of finite lattices with order and join-irreducible preserving quotient maps}\label{Section_last}
In this section we are going to define a category in which the sequence $\langle\{\O(P_n)\}_{n \in \N},\{\hat{p}^n_k\}_{k<n}\rangle$ of order ideals related to $\langle\{P_n\}_{n \in \N},\{p^k_n\}_{k<n}\rangle$ (the \text{\Fraisse} sequence \eqref{SequanceOfUniwersalProjectivePosets})  is a \text{\Fraisse} sequence itself.
To define that category we need the notion of isomorphic arrows: we say that a pair $(g_1,g_2)$ is the arrow from $f:A \to B$ to $f':A' \to B'$, if we the following diagram
\begin{equation*}
    \begin{tikzcd}
     A \arrow [d, "g_1"] \arrow[r, "f"] & B \arrow[d, "g_2"] \\
     A' \arrow[r, "f'"] & B' \\
    \end{tikzcd}
\end{equation*}
 commutes.  
The composition of these arrows is  component-wise, i.e. $(g_1, g_2)\circ (h_1, h_2)=(g_1 \circ h_1, g_2 \circ h_2)$. 
We say that $(g_1, g_2)$ is an isomorphism, if there are $(h_1, h_2)$, such that $(g_1 \circ h_1, g_2 \circ h_2)=(\id_A, \id_B)$ and $(h_1 \circ g_1, h_2 \circ g_2)=(\id_{A'}, \id_{B'})$.
If there is isomorphism between $f:A \to B$ and $f': A' \to B'$, then we say $f$ and $f'$ are isomorphic. Clearly, if $g_1$ and $g_2$ are isomorphisms, then $(g_1,g_2)$ is also isomorphism.
\begin{thm}\label{Join_preserving_mappings_that_preserve_JIRRE_Are_IQM's}
Let $L,T$ be finite distributive lattices, $p: L \to T$ be a mapping satisfying the following conditions:
\begin{itemize}
    \item[(i)] $p_{|\J(L)}$ is a quotient mapping;
    \item[(ii)] $p(\J(L))=\J(T)$;
    \item[(iii)] $p$ is a quotient mapping that is join-preserving.
\end{itemize}
Then $\hat{p}_0$ is isomorphic to $p$, where $p_0=p_{|\J(L)}$

\end{thm}
\begin{proof}
We will show that the following diagram 
\begin{equation*}
    \begin{tikzcd}
    T \arrow[d,"\mathbf{B}_T"] & L \arrow[l, "p"'] \arrow[d, "\mathbf{B}_L"] \\
    \O(\J(T)) & \arrow[l, "\hat{p}_0"'] \O(\J(L)) \\
    \end{tikzcd}
\end{equation*}
is commutative,
where $\mathbf{B}_T$ and $\mathbf{B}_L$ are the isomorphisms from Birkhoff's Theorem (Theorem \ref{Birkhoff's_Theorem}).
Take any $x \in \J(L)$. Then $\mathbf{B}_T(p(x))=\dol p(x)$ and $\hat{p}_0(\mathbf{B}_L(x))=\hat{p}_0(\dol x)= \dol p(x)$, so the diagram commutes for join-irreducible elements.

Now let $a\in L\setminus\J(L)$ be non-zero. Then $\mathbf{B}_L(a)=\bigcup_{i=1}^n \dol a_i$, thus $\hat{p}_0(\mathbf{B}_L(a))=\bigcup_{i=1}^n \dol p(a_i)$. See that $\mathbf{B}_T(p(a))=\mathbf{B}_T(p(\mathbf{B}_L^{-1}(\bigcup_{i=1}^n \dol a_i)))=\mathbf{B}_T(p(\bigvee_{i=1}^n a_i))=\bigcup_{i=1}^n 
\mathbf{B}_T(p(a_i))=\bigcup_{i=1}^n \dol p(a_i)$.

Now see that any non-zero elements of $L$ are mapped to non-zero elements of $T$ by $p$, and $p$ is onto, thus $p(\mathbf{0}_L)=\mathbf{0}_T$. Hence $\hat{p}_0(\mathbf{B}(\mathbf{0}_L))=\hat{p}_0(\emptyset)=\emptyset = \mathbf{B}(\mathbf{0}_T)=\mathbf{B}(p(\mathbf{0}_L))$.
This proves that the diagram commutes. Since $\mathbf{B}$ is the isomorphism, $p$ and $\hat{p}_0$ are isomorphic.
\end{proof}

Theorem \ref{Join_preserving_mappings_that_preserve_JIRRE_Are_IQM's} shows that we can identify mappings satisfying conditions (i)--(iii)  with induced quotient mappings. Conditions (i)--(iii) allows us to have the inner definition of induced quotient mappings (up to isomorphism). Using this we define a category of finite distributive lattices with mappings isomorphic to induced quotient mappings. From now on we do not distinguish mappings from that isomorphic to them.

Let $V$ be a countable class of all finite distributive lattices. Formally, one can define $V$ as a set of all order ideals of finite posets on $\N$. Let $A$ be the class of all induced quotient mappings between finite distributive lattices; we say that $f$ is an arrow from $A$ to $B$, provided $f$ is an induced quotient mapping from $B$ onto $A$. Then $\dom f$ is the range of $f$ and $\cod f$ is the domain of $f$.   
The composition $g\bullet f$ of two arrows, $f$ from $A$ to $B$ and  $g$ from $B$ to $C$, is $f \circ g:C\to A$, where $\circ$ is the usual mappings composition. Then $\L:=\langle V,A,\dom,\cod,\bullet \rangle$ is a category.

\subsection{$\O(P_n)$ is a \text{\Fraisse} sequence in $\L$}
Let $\langle \{P_n\}_{n \in \N},\{p^n_k\}_{k < n}\rangle$ be the inverse system defined in Section \ref{UniversalProjectivePoset}. We will prove that $\langle \{\O(P_n)\}_{n \in \N},\{\hat{p}^n_k\}_{k < n}\rangle$ is a \text{\Fraisse} sequence in $\L$.

Let $L=\O(P)$ be a finite, distributive lattice. We know there exists $n \in \N$ and quotient mapping $q:P \to P_n$. Then $\hat{q}$ is an induced quotient mapping from $L$ onto $\O(P_n)$. This shows that the \hyperlink{UProperty}{(U)} condition holds.

Take any $n \in \N$ and an induced quotient mapping $h:L\to \O(P_n)$. There is a quotient $f:\J(L)\to P_n$ with $h=\hat{f}$. 
There is $m>n$ and $g$ such that the following diagram
\begin{equation*}
    \begin{tikzcd}
        \O(P_n) & \arrow[l, "h=\hat{f}"] L=\O(\J(L)) & \arrow[l, "\hat{g}"] \O(P_m) \\
        P_n \arrow[u] & \J(L) \arrow[u] \arrow[l, "f"] & P_m \arrow[u] \arrow[l, "g"] \\
    \end{tikzcd}
\end{equation*}
commutes, where $g \circ f=p_n^m$ and up arrows are canonical poset embeddings  into order ideals. By Lemma \ref{InducedMappingsComposition}, $\hat{g} \circ \hat{f}=\hat{p}_n^m$, which shows property \hyperlink{AProperty}{(A)} and finishes the proof.

\subsection{A different approach to the inverse limit of $\langle \{\O(P_n)\}_{n \in \N},\{\hat{p}^n_k\}_{k < n}\rangle$.}
Let $\langle \{P_n\}_{n \in \N},\{p^n_k\}_{k\leq n}\rangle$ be the inverse system defined in Section \ref{UniversalProjectivePoset}, which inverse limit is $\mathbb{P}$. We know that $P_n$ is isomorphic to $\dot\bigcup_{i<2\cdot 4^{n-1}}\mathbf{2}$. Moreover for any two $x,y\in P_n$ with $x<_n y$, there are $x_i,y_i\in P_{n+1}$, $i=1,2,3,4$, with $x_i<_{n+1} y_i$, and such that
\begin{itemize}
    \item $p_{n}(x_1)=p_n(x_2)=x$, $p_n(y_1)=p_n(y_2)=y$, that is $p_n$ maps both $\mathbf{2}$-components $\{x_1,y_1\}$ and $\{x_2,y_2\}$ onto a $\mathbf{2}$-component $\{x,y\}$;
    \item $p_{n}(x_3)=p_n(y_3)=x$, that is $p_n$ maps a $\mathbf{2}$-component $\{x_3,y_3\}$ onto $x$;
    \item $p_{n}(x_4)=p_n(y_4)=y$, that is $p_n$ maps a $\mathbf{2}$-component $\{x_4,y_4\}$ onto $y$.
\end{itemize}
Note that exactly those properties of $\langle (P_n),(p^n_k)_{k\leq n}\rangle$ where needed to prove that it is a Fraisse sequence in category of finite posets, see Theorem ... and its proof.  

Let $T_n=\{x\in\{0,1,2,3\}^n:x(0)=0,1\}$. Note that $\vert T_n\vert = 2\cdot 4^{n-1}$ is the same as the number of $\mathbf{2}$-components of $P_n$. We use $T_n$ to represent $\mathbf{2}$-components of $P_n$ in the following way. $T_1=\{0,1\}$ and 0 and 1 represents two $\mathbf{2}$-components of $P_1$. For $c\in T_n$ let $\{x,y\}$ with $x<y$ be a $\mathbf{2}$-component of $P_n$. Then $c\hat{\;\;}0$ and $c\hat{\;\;}1$ represent those two $\mathbf{2}$-components of $P_{n+1}$ which are mapped onto $\{x,y\}$, $c\hat{\;\;}2$ represents $\mathbf{2}$-component of $P_{n+1}$ which is mapped onto $x$, and $c\hat{\;\;}3$ represents $\mathbf{2}$-component of $P_{n+1}$ which is mapped onto $y$.

Now, let us represent $\mathcal{O}(P_n)$ as $\{0,1,2\}^{T_n}$. The idea is the following. Any down-set can be represent as a union of principal down-sets. Any principal down-set in $P_n$ is of the form $\{x\}$ or $\{x,y\}$ where $c:=\{x,y\}$ with $x<y$ is a $\mathbf{2}$-component. Now, a down-set $A\in\mathcal{O}(P_n)$ can be represent as a function $f:T_n\to\{0,1,2\}$ such that $f(c)=2$ is $A$ contains $\{x,y\}$, $f(c)=1$ if $A$ contains $\{x\}$ but does not $\{x,y\}$, and $f(c)=0$ if $A$ does not contain $\{x\}$.    

For every $f,g \in \{0, 1, 2\}^{T_n}$ we define $(f \vee g)(c)=\max\{f(c), g(c)\}$ and $(f \wedge g)(c)=\min\{f(c), g(c)\}$ for every $c\in T_n$. Clearly $\{0, 1, 2\}^{T_n}$ is a lattice, with $\vee$ and $\wedge$ as join and meet, respectively.

Note that the set of all join-irreducible elements $\J(\{0, 1, 2\}^{T_{n}})$ in $\{0, 1, 2\}^{T_n}$ consists of those $f\in \{0, 1, 2\}^{T_{n}}$ for which there is exactly one $c\in T_{n}$ with $f(c)\neq 0$; consequently $f(c) \in \{1, 2\}$.  
For given $f \in \J(\{0, 1, 2\}^{T_{n}})$ 
we denote the unique $c\in T_{n}$ with $f(c)\neq 0$ by $c_f$. We also say that $f$ is supported by $c$ if $c=c_f$.

The following fact formalizes the connection between $\O(P_n)$ and $\{0, 1, 2\}^{T_m}$.
\begin{lem}
	Let $\psi:\mathcal{O}(\dot\bigcup_{\mathbf{x} \in T_n}\mathbf{2}) \to \{0, 1, 2\}^{T_n}$ be defined step by step as follows
	\begin{itemize}
	\item $\psi(\emptyset)=0$ (zero function)
	
	\item For every $\mathbf{2}$-component $\{x, y\}$ with $x < y$ and $c_0\in T_n$ which represents that  $\mathbf{2}$-component we put 
	\[\psi( \dol x)(c)= \begin{cases}
0, & c \neq c_0 \\
1, & c=c_0, 				
\end{cases}\;\;\text{ and }\;\;
\psi( \dol y)(c)= \begin{cases}
0, & c \neq c_0 \\
2, & c=c_0, 				
\end{cases}
\]
	
	\item For every non-empty join-irreducible element $A\in\mathcal{O}(\dot\bigcup_{\mathbf{x} \in T_n}\mathbf{2})$ and its canonical decomposition $\{\dol a_i: i<k\}$ we define $\psi(A)=\bigvee_{i=1} ^k \psi(\dol a_i)$
	\end{itemize}
	Then $\psi$ is a lattice isomorphism.
\end{lem}
We will call $\psi$ the canonical isomorphism between $\mathcal{O}(\dot\bigcup_{\mathbf{x} \in T_n}\mathbf{2})$ and $\{0, 1, 2\}^{T_n}$.
\begin{proof}
By Lemma \ref{Lattice_And_Order_Isomorphisms} we only need to show that $\psi$ is an order-isomorphism.

Firstly we show that it is join-preserving.

See that for every $x,y \in \dot\bigcup_{\mathbf{x} \in T_n}\mathbf{2}$ we have that $\psi(\dol x \cup \dol y)=\psi(\dol x) \vee \psi(\dol y)$.
To see it we need to consider the following two cases:
\begin{itemize}
	\item If $\dol x \subsetneq \dol y$ then $x <_n y$ by Lemma \ref{DownsetLemma1}. Then $\psi(\dol x \cup \dol y)=\psi( \dol y)=g$ and $g(c_g)=2$. See that $\psi(\dol x)=f$ where $f$ is such that $f(c_g)=1$ and $f(c)=0$ for remaining $c$'s. Then $f \vee g=g$, and consequently $\psi(\dol x \cup \dol y)=g=f \vee g=\psi(\dol x) \vee \psi(\dol y)$.
	\item If $\dol x \parallel \dol y$, then immediately from the definition of $\psi$ we have $\psi(\dol x \cup \dol y)=\psi(\dol x ) \vee \psi(\dol y)$.
\end{itemize}
By induction it can be proven that $\psi(\bigcup_{i=1} ^k \dol x_i)=\bigvee_{i=1} ^k \psi(\dol x_i)$. This implies that $\psi(A \cup B)= \psi(A) \vee \psi(B)$ (cf. proof of Theorem \ref{Posets_and_O(P)}). Thus $\psi$ is join-preserving. Consequently $A\subset B$ implies $\psi(A)\leq\psi(B)$   for any $A,B\in \O(\dot\bigcup_{\mathbf{x} \in T_n}\mathbf{2})$.

Now we will prove that $\psi(A)\leq\psi(B)$ implies $A\subset B$ for any $A,B\in \O(\dot\bigcup_{\mathbf{x} \in T_n}\mathbf{2})$.
We start proving it for principal down-sets. Let $x ,y \in P_n$ be such that $\psi(\dol x) \leq \psi(\dol y)$.
Let $f=\psi(\dol x),g=\psi(\dol y)$. Then $f \leq g$  means that $c_f=c_g$. Thus $x, y$ are in one $\mathbf{2}$-component of $P_n$. Since $f(c_f) \leq g(c_g)$, then $x \leq_{P_n} y$. Consequently $\dol x \subset \dol y$ by Lemma \ref{DownsetLemma1}.

Now we consider the general case. Let $A=\bigcup_{i=1} ^k \dol a_i, B=\bigcup_{i=1} ^l \dol b_i \in \O(\dot\bigcup_{\mathbf{x} \in T_n}\mathbf{2})$ and assume that $\bigvee_{i=1}^k \psi(\dol a_i)=\psi(A) \leq \psi(B)=\bigvee_{i=1}^l \psi(\dol b_i)$. By Lemma \ref{Lemma_On_Irreducible_Elements} for every $\dol a_i$ we can find $\dol b_{j_i}$ such that $\psi(\dol a_i) \leq \psi(\dol b_{j_i})$. Then $\dol a_i \subset \dol b_{j_i}$ for every $i \in \{1, \dots, k\}$. Thus $A \subset B$. 

The fact that $\psi$ is an injection follows from preserving order both ways.

Now we prove that $\psi$ is surjective.
Take any $f \in \{0, 1, 2\}^{T_n}$. Then $f=\bigvee_{i=1} ^k f_i$ where $f_1, \dots, f_k \in \J(\{0, 1, 2\}^{T_n})$. For every $f_i$ we can find $\dol x_i$ such that $x_i$ is in $\mathbf{2}$-component represented by $c_{f_i}$. Depending on value of $f(c_{f_i})$ we take $x_i$ to be lower or upper element of $\mathbf{2}$-component $c_{f_i}$. Then $\psi(\dol x_i)=f_i$. Therefore $\psi(\bigcup_{i=1} ^k \dol x_i)=\bigvee_{i=1} ^k f_i=f$.
\end{proof}
 Let $\psi$ be the canonical isomorphism between $\O(P_n)$ and $\{0, 1, 2\}^{T_n}$, and $\phi$ be the canonical isomorphism between $\O(P_{n+1})$ and $\{0, 1, 2\}^{T_{n+1}}$.

We will define $q_n ^{n+1}:\{0,1,2\}^{T_{n+1}}\to\{0,1,2\}^{T_{n}}$ such that diagram
\begin{equation}\label{Different_Form_Of_IQM}
\begin{tikzcd}
\O(P_n) \arrow[d, "\psi"]  & \O(P_{n+1}) \arrow[l, "\widehat{p}_n ^{n+1}"']  \arrow[d, "\phi"] \\
\{0, 1, 2\}^{T_n}					& \{0, 1, 2\}^{T_{n+1}}  \arrow[l,"q_n ^{n+1}"]  \\
\end{tikzcd}
\end{equation}
commutes.
As $\mathcal{O}(P_n)$ is isomorphic to $\{0,1,2\}^{T_n}$, we can express every non-zero, non join-irreducible element of $\{0,1,2\}^{T_n}$ as a join of join-irreducible, non-comparable elements of $\{0,1,2\}^{T_n}$. We only need to define $q_n ^{n+1}$ for join-irreducible elements of $T_{n+1}$ and expand it using the aforementioned property. 
Let $f \in \J(\{0, 1, 2\}^{T_{n+1}})$. Then $f(c)\neq 0\iff c=c_f$. We define $q_n ^{n+1}(f)=g$ where 
\[g({c_f}\upharpoonright n)= \begin{cases}
f(c_f) & \text{ if } c_f(n) \in \{0, 1\} \\
1 & \text{ if } c_f(n)=2\\
2 & \text{ if } c_f(n)=3
\end{cases}\;\;\text{ and }\;\;g(c)=0\iff c\neq c_f\upharpoonright n.
\]
Then $g \in \J(\{0, 1, 2\}^{T_{n}})$.
Additionally, for $f \equiv 0$ we put $q_n ^{n+1}(f) \equiv 0 \in \{0, 1, 2\}^{T_n}$.

In general, for $f \in \{0 ,1, 2\}^{T_{n+1}}$ we have \[f=\bigvee_{c\in T_{n+1}} f_c=\max_{c\in T_{n+1}}f_c\] where $f_c$ is join-reducible supported by $c$. This implies that 
$$q_n ^{n+1}(f)(c)=\max \big\{f(c\hat{\;\;}0), f(c\hat{\;\;}1), \min \{f(c\hat{\;\;} 2), 1\}, 2\cdot\min\{f(c\hat{\;\;}3), 1\}\big\}$$
for every $f\in\{0,1,2\}^{T_n}$ and $c\in T_n$.

We have proven the following.

\begin{prop}
The lattice $\mathbb{O(P)}$ is isomorphic to the inverse limit of inverse system $\langle\{\{0,1,2\}^{T_n}\}_{n \in \N},\{q^n_k\}_{k<n}\rangle$.
\end{prop}
We hope that this representation of $\mathbb{O(P)}$ will allow to study its properties in future.

\bibliography{bibliografia.bib}
\bibliographystyle{siam}

\end{document}